\documentclass[11pt]{article}

\usepackage[utf8]{inputenc}
\usepackage[T1]{fontenc}
\usepackage[ngerman,english]{babel}
\usepackage{lmodern}

\usepackage{amsmath}
\usepackage{amsthm}
\usepackage{amsfonts}
\usepackage{amssymb}
\usepackage{amscd}

\usepackage{fancyhdr}
\usepackage{geometry}
\usepackage{setspace}
\usepackage{color}
\usepackage{tocloft}

\usepackage{makeidx}
\usepackage{graphicx}
\usepackage{tabularx}
\usepackage{epsfig}

\usepackage{eurosym}
\usepackage{enumerate}

\usepackage[colorlinks,citecolor=red,linkcolor=blue]{hyperref}

\theoremstyle{definition}
\newtheorem{definition}{Definition}[section]

\theoremstyle{plain}
\newtheorem{theorem}[definition]{Theorem}
\newtheorem{corollary}[definition]{Corollary}
\newtheorem{lemma}[definition]{Lemma}
\newtheorem{proposition}[definition]{Proposition}

\theoremstyle{remark}

\newtheorem{example}[definition]{Example}

\newcommand{\N}{{\mathbb N}}

\newcommand{\R}{{\mathbb R}}

\newcommand{\E}{{\mathbb E}}

\newcommand{\err}{{\mathrm{err}}}
\newcommand{\disc}{{\mathrm{disc}}}

\newcommand{\tr}{\!^\top\!}

\newcommand{\lin}{\mathrm{span}}
\newcommand{\integ}{\mathrm{int}}
\newcommand{\id}{\mathrm{id}}

\newcommand{\cA}{{\mathcal A}}

\newcommand{\cD}{{\mathcal D}}

\newcommand{\cH}{{\mathcal H}}
\newcommand{\cL}{{\mathcal L}}
\newcommand{\cP}{{\mathcal P}}

\newcommand{\cK}{{\mathcal K}}
\newcommand{\cJ}{{\mathcal J}}
\newcommand{\cO}{{\mathcal O}}

\newcommand{\x}{{\boldsymbol{x}}}
\newcommand{\y}{{\boldsymbol{y}}}

\newcommand{\e}{{\boldsymbol{e}}}

\renewcommand{\k}{{\boldsymbol{k}}}
\renewcommand{\j}{{\boldsymbol{j}}}

\renewcommand{\r}{{\boldsymbol{r}}}
\renewcommand{\t}{{\boldsymbol{t}}}
\newcommand{\p}{{\boldsymbol{p}}}
\newcommand{\w}{{\boldsymbol{w}}}
\newcommand{\ellb}{{\boldsymbol{\ell}}}
\newcommand{\gammab}{{\boldsymbol{\gamma}}}
\newcommand{\alphab}{{\boldsymbol{\alpha}}}
\newcommand{\omegab}{{\boldsymbol{\omega}}}
\newcommand{\betab}{{\boldsymbol{\beta}}}
\newcommand{\epsilonb}{{\boldsymbol{\epsilon}}}
\newcommand{\0}{{\boldsymbol{0}}}
\newcommand{\1}{{\boldsymbol{1}}}

\newcommand{\linpart}{{\boldsymbol b}}

\newcommand{\TODO}[1]{}
\newcommand{\EXCLUDE}[1]{}

\allowdisplaybreaks

\begin{document}

\title{High-dimensional integration on $\R^d$, weighted Hermite spaces,
and orthogonal transforms}

\author{Christian Irrgeher\thanks{Christian Irrgeher, Institute of Financial Mathematics, Johannes Kepler University Linz, Altenbergerstra{\ss}e 69, A-4040 Linz, Austria. e-mail: {\tt christian.irrgeher@jku.at}}~~and Gunther Leobacher\thanks{Gunther Leobacher, Institute of Financial Mathematics, Johannes Kepler University Linz, Altenbergerstra{\ss}e 69, A-4040 Linz, Austria. e-mail: {\tt gunther.leobacher@jku.at}.}}

\maketitle

\begin{abstract}
It has been found empirically that quasi-Monte Carlo methods
are often efficient for very high-dimensional
problems, that is, with dimension in the hundreds or even thousands. 
The common explanation for this surprising fact is that those functions
for which this holds true 
behave rather like low-dimensional functions in that only few of the 
coordinates have a sizeable influence on its value. However, this statement
may be true only after applying a suitable orthogonal transform to the
input data, like utilizing the Brownian bridge construction or principal 
component analysis construction.

We study the effect of general orthogonal transforms on functions on $\R^d$ 
which are elements of certain 
weighted reproducing kernel Hilbert spaces. The notion of
Hermite spaces is defined and it is shown that there are examples which admit
tractability of integration. We translate the action of the orthogonal
transform of $\R^d$ into an action on the Hermite coefficients and we 
give examples where orthogonal transforms have
a dramatic effect on the weighted norm, thus providing an explanation for 
the efficiency of using suitable orthogonal transforms.
\end{abstract}




\section{Introduction}\label{sec:intro}

Many important problems from quantitative finance  and other applications of
probability theory can be written as expected values of functions depending on
the path of a Brownian motion, i.e.  \begin{align}\label{eq:mainproblem}
\E(g(B))=\ ?
\end{align}
where $B=(B_t)_{t\in[0,T]}$ is an $m$-dimensional standard Brownian motion with
index set $[0,T]$ and $g$ is a  function such that the expected value is
well-defined. The most prominent example from finance is the celebrated
Black-Scholes equation for option pricing. In its general form it states that
the problem of determining the value of a European-style contingent claim --
like a European or Asian option -- in an arbitrage-free market model driven by
a Brownian motion, can be formulated as the computation of an expected value of
the form \eqref{eq:mainproblem}, see e.g.\ Bj{\"o}rk \cite{bjork}.

In most cases of practical interest these problems cannot be solved in closed
form, and thus it is necessary to approximate \eqref{eq:mainproblem}
numerically. A very versatile way of doing this is by using Monte Carlo,  that
is, a large number of random paths of the Brownian motion is generated and the
function value of $g$ over this sample is averaged. Usually, the actual Monte
Carlo simulation requires the discretization of the Brownian path. One thus
uses a randomized algorithm $Q_{n,d}$, where $d$ denotes the dimension of the
discretized problem and $n$ corresponds to the number of (discretized) paths
over which the average is taken. 

Also if a deterministic algorithm, like quasi-Monte Carlo, is to be applied,
the problem has first to be discretized. For that, we use a time grid of $d_0$
equidistant points to discretize the time interval $[0,T]$ such that we get a
discrete $m$-dimensional Brownian path
$(B_{1/d_0},B_{2/d_0},\ldots,B_{d_0/d_0})$ with $B_{i/d_0}\in\R^m$ for
$i=1,\ldots,d_0$. Moreover, we have to choose a $d=m\times d_0$-dimensional
function $g_{d_0}$ depending on the discrete Brownian path
$(B_{1/d_0},B_{2/d_0},\ldots,B_{d_0/d_0})$ such that
$\E(g(B))\approx\E(g_{d_0}(B_{1/d_0},B_{2/d_0},\ldots,B_{d_0/d_0}))$. So the
problem is approximated by the multivariate integration problem
$\E(g_{d_0}(B_{1/d_0},B_{2/d_0},\ldots,B_{d_0/d_0}))$ of dimension $d$ which can
be solved, e.g., by quasi-Monte Carlo integration, see Dick and Pillichshammer
\cite{dick}. Note that, here the integration problem for some
function $f$ is given over $\R^d$ and thus, a quasi-Monte Carlo method 
for computing this integral is given by an equally weighted quadrature rule 
$Q_{n,d}(f,\cP)=\frac{1}{n}\sum_{i=1}^{n}f(\x_i)$
with deterministic point set $\cP=\{\x_1,\ldots,\x_n\}\subset\R^d$. 

In contrast to Monte Carlo, the method with which the discrete Brownian paths
are generated makes a difference for quasi-Monte Carlo methods. There are three
classical Brownian path construction methods: the forward method (or
step-by-step method), the Brownian bridge construction (a.k.a.
L\'evy-Ciesielski construction) and the principal component analysis (PCA)
construction. In Moskowitz and Caflisch \cite{moskowitz} the convergence of QMC
integration is dramatically improved for some examples from finance by using
the Brownian Bridge construction. Similar results are presented by Acworth et
al.\ \cite{acworth} for the PCA construction. However, Papageorgiou \cite{papa}
shows that there are problems for which the forward construction gives faster
convergence than the Brownian bridge construction. He further shows that every
linear construction method for a one-dimensional Brownian path corresponds to a
unique orthogonal transform of the $\R^d$. Subsequently, he formulates an
equivalence principle that roughly states that every construction that is good
for one problem is bad for another problem. Wang and Sloan \cite{wang} give a
more general version of that equivalence principle.

Due to the equivalence principle, the choice of the orthogonal transform should
depend on the integration problem, as given by the function $g$. Leobacher
\cite{leobacher}, proposes to restrict this search to  transforms which can be
applied using $\mathcal{O}\left(d\log(d)\right)$ operations. In Imai and Tan
\cite{imaitan07} as well as in Irrgeher and Leobacher \cite{il12,il13} methods
are presented to find, for a given problem, an (in some sense) optimal
orthogonal transform. In both approaches the integration problem is linearized
and the ``optimal'' transform is determined for the linear problem. Imai and
Tan use a Taylor expansion of order $1$ for the linearization step whereas in
Irrgeher and Leobacher \cite{il12} a linear regression is performed. The idea
behind both methods is to minimize the effective dimension, as defined by
Caflisch, Morokoff, and Owen \cite{caflisch97}. 

While the concept of effective dimension has its merits in explaining the
effectiveness of deterministic methods for high-dimensional problems, it also
does have some drawbacks. One of those drawbacks is that it does not interact
smoothly with orthogonal transforms. For example, consider some continuous
function $f:\R\longrightarrow\R$ with $\int_\R f(x)^2\varphi(x)dx<\infty$ and
$\int_\R f(x)\varphi(x)dx=0$, where
$\varphi(x)=(2\pi)^{-\frac{1}{2}}\exp(-\frac{x^2}{2})$ is the standard normal
density. Then the function $g:\R^2\longrightarrow\R$ defined by
$g(x_1,x_2):=f(x_1)$ is continuous with $\int_{\R^2}
g(x_1,x_2)^2\varphi(x_1)\varphi(x_2)dx_1dx_2<\infty$ and $\int_{\R^2}
g(x_1,x_2)\varphi(x_1)\varphi(x_2)dx_1dx_2=0$.  The ANOVA decomposition of $g$
is given by $g_{\{1\}}=f$, $g_\emptyset=g_{\{2\}}=g_{\{1,2\}}\equiv 0$. On the
other hand, consider $\tilde g=g\circ U$, where $U$ is some orthogonal
transform of the $\R^2$. Then $\tilde g(x_1,x_2)=f(c x_1+s x_2)$ with
$c^2+s^2=1$, and therefore we have $\tilde g_{\{1,2\}}=\tilde g$, and $\tilde
g_\emptyset=\tilde g_{\{1\}}=\tilde g_{\{2\}}\equiv 0$, for every choice of
$c,s$ with $|c|\ne 1$ and $|s|\ne 1$. Thus even  the slightest rotation
transforms a function of effective dimension 1 into one with effective
dimension 2 (both in the truncation and in the superposition sense). It is not
hard to construct multidimensional examples where a slight orthogonal transform
effects an arbitrary change in effective dimension.

We have a competing explanation for the efficiency of QMC, namely the theory of
weighted spaces as proposed by Sloan and Wo\'zniakowski \cite{sw98}. They
introduce weighted norms on Sobolev spaces that assign different degrees of
importance to different coordinates. The idea is related to the concept of
effective dimension in that both concepts concentrate on problems for which
only relatively few input parameters are really important. From the point of
view of orthogonal transforms it is now interesting to ask whether one can
concatenate the original function with an orthogonal transform in a way that
makes the weighted norm of the resulting function -- and thus the integration
error -- small. Before we proceed with this program, we need to go back one
step. 

We are interested in analyzing the error which occurs by approximating the
expected value $\E(g(B))$ by using a QMC algorithm $Q_{n,d}$, where $d$ gives
the dimension of the discretized space and $n$ is the number of integration
nodes. The error is given by \begin{align*}
\err:=\big\vert\E(g(B))-Q_{n,d}\big\vert.
\end{align*}
As mentioned earlier, the function $g$ has to be approximated by a $d=m\times
d_0$-dimensional function $g_{d_0}$ which depends on a discrete $m$-dimensional
Brownian path. Then there further exists a function $f_d:\R^d\longrightarrow\R$
such that $f_d(X)=g_{d_0}\big((B_{T/d_0},B_{2T/d_0},\ldots,B_T)\big)$ where
$X=(X_1,\ldots,X_d)$ is a standard Gaussian vector. Hence, \begin{align}
\err&=\big\vert\E(g(B))-\E(f_d(X))+\E(f_d(X))-Q_{n,d}\big\vert\nonumber\\
&\leq \underset{=:\err_{\disc}}{\underbrace{\big\vert\E(g(B))-\E(f_d(X))\big\vert}}
+\underset{=:\err_{\integ}}{\underbrace{\big\vert\E(f_d(X))-Q_{n,d}\big\vert}}\label{eq:errorbound1}\,.
\end{align}

$\err_\disc$ is called the \emph{discretization error} and $\err_\integ$ is called the \emph{integration error}. That means, we can bound
the total error by the sum of the error coming from the discretization and the
error which comes from the QMC integration. Now it is obvious that we prefer to
approximate the function $g$ by a function $g_{d_0}$ or $f_d$, respectively,
such that the discretization error $\err_\disc$ becomes small. But it is
equally important that $\err_\integ$ does not explode with growing $d$.

An important practical example is the case where $g(B)$ is described by a
solution of a stochastic differential equation (SDE), e.g., $g(B)=\psi(X_T)$
where $X$ is a stochastic process that solves \begin{align*}
X_0&=x_0\\
dX_t&=a(t,X_t)dt+b(t,X_t)dB_t
\end{align*}
where $a$ is a $\R^k$-valued function and $b$ is a $\R^{k\times m}$-valued
function. There are many different discretization methods in the theory of
stochastic differential equation which can be applied to $g(B)$. The two
best-known are the Euler-Maruyama method and the Milstein method, but there are
also higher-order Runge-Kutta methods. For more information about these
discretization methods we refer to Kloeden and Platen \cite{kloeden} which also
provides an extensive analysis of the convergence of these methods. In
\cite{kloeden} it is shown that, under mild conditions on the coefficient
functions $a$ and $b$, as well as  the ``payoff function'' $\psi$, we have an
upper bound on the discretization error of the form
\begin{align}\label{eq:discerrorbound}
\err_{\disc}=\big\vert\E(g(B))-\E(f_d(X))\big\vert\leq c d^{-q}\,,
\end{align}
where $c>0$ is constant and $q>0$ is the convergence rate which depends on the
discretization method used and on the function $\psi$. In Glasserman
\cite[Chapter 6]{glasserman} it is discussed how these discretization methods
are applied to problems coming from finance.  So for this kind of problems
\eqref{eq:discerrorbound} provides us with an estimate on $\err_\disc$ in
inequality \eqref{eq:errorbound1}.

The emphasis of this paper lies on the analysis of the integration error
$\err_\integ$ to get an upper bound on the total error. Moreover, we are
interested in the behavior of the integration error with respect to $d$,
because we know that the discretization error can only be reduced by increasing
the dimension $d$. This will lead us to the study of QMC-tractability.  In
particular, we want to know how orthogonal transforms affect the weighted norm
and whether the growth of the complexity of integration can be moderated by the
use of suitable orthogonal transforms. To that end we consider special spaces
of integrable functions on the $\R^d$ which we call {\em Hermite spaces}. Those
spaces are spanned by Hermite polynomials, which enjoy a certain invariance
under orthogonal transforms.  It turns out that this class of spaces contains
examples of weighted spaces for which tractability can be proven and which are
sufficiently rich to contain interesting functions.

The remainder of the paper is organized in the following way. We recall the basic facts about Gaussian measures and Hermite polynomials in Section \ref{sec:hermite} and we discuss the Hermite expansion of functions. In one of the main parts of the paper, namely in Section \ref{sec:integration}, we introduce Hermite spaces, present basic properties and discuss multivariate integration in these spaces. Furthermore, we investigate the tractability of QMC methods in these spaces and we discuss what kind of functions belong to Hermite spaces. In the second main part, Section \ref{sec:orthogonal}, we show that orthogonal transformations can be used to reduce the weighted norm of a function,  thus improving the convergence of a QMC algorithm. We further give a representation of the operator $\cA_U$ on the Hermite space that is induced by an orthogonal transform $U$ of the $\R^d$ in terms of direct sums of tensor powers of $U$. Finally, Section \ref{sec:conclusion} summarizes the most important results and states some open problems.

\section{Hermite polynomials and Hermite expansion}\label{sec:hermite}

In this section we recall some basic definitions and results concerning  Gaussian measure, the Hermite polynomials and the Hermite expansion. 

\subsection{Gaussian measure}\label{ssec:gauss}

\begin{definition}
The \emph{standard Gaussian measure} is a Borel probability measure on the $\R^d$ with density $\varphi_d:\R^d\longrightarrow\R$ given by
\begin{align*}
\varphi_d(\x)=(2\pi)^{-\frac{d}{2}}\,e^{-\frac{\x\tr\x}{2}}
\end{align*}
with respect to the $d$-dimensional Lebesgue measure.
\end{definition}

The notation $\x\tr\x$, as it is used in the above definition, denotes the inner product in the Euclidean space, i.e., for any $\x,\y\in\R^d$ we have $\x\tr\y=\sum_{k=1}^{d}x_ky_k$ with $\x=(x_1,\ldots,x_d)$ and $\y=(y_1,\ldots,y_d)$. We will exclusively work with the standard Gaussian measure. A more general discussion about Gaussian measures can be found in Bogachev \cite{bogachev}. Furthermore, we will write $\varphi$ instead of $\varphi_1$ in the one-dimensional case.

We say that a measurable function $f:\R^d\longrightarrow\R$ is Gaussian square-integrable if 
\begin{align*}
\int_{\R^d}f(\x)^2\varphi_d(\x)d\x<\infty\,,
\end{align*}
and we denote the linear space of Gaussian square-integrable functions by
$\cL^2(\R^d,\varphi_d)$.

Moreover, we say that two functions $f$ and $g$ are equivalent if $f=g$ almost everywhere. Then, the space of all equivalence classes of Gaussian square-integrable functions on the $\R^d$, denoted by $L^2(\R^d,\varphi_d)$, forms a Hilbert space with inner product
\begin{align*}
\langle f,g\rangle_{L^2(\R^d,\varphi_d)}=\int_{\R^d}f(\x)g(\x)\varphi_d(\x)d\x.
\end{align*}
The corresponding norm is induced by the inner product, i.e., $\|f\|_{L^2(\R^d,\varphi_d)}=\sqrt{\langle f,f\rangle_{L^2(\R^d,\varphi_d)}}$.

\subsection{Hermite polynomials}\label{ssec:hermite}

Before we introduce the multivariate Hermite polynomials, we recall multi-index notation: a $d$-dimensional {\em multi-index} is a $d$-tuple $\k=(k_1,\ldots,k_d)$ with nonnegative integer entries. We use the following conventions for $\k,\j\in\N_0^d$ and $\x\in\R^d$:
\begin{enumerate}[(i)]
\item $\k\leq \j\quad\Longleftrightarrow\quad k_i\leq j_i$ for all $i=1,\ldots,d$;
\item $\k+\j=(k_1+j_1,\ldots,k_d+j_d)$;
\item $|\k|=k_1+\ldots+k_d$;
\item $\k!=k_1!\cdots k_d!$;
\item $\x^{\k}=x_1^{k_1}\cdots x_d^{k_d}$.
\end{enumerate}

There are some slightly different definitions of Hermite polynomials in the literature which are all related, see, e.g., Bogachev \cite{bogachev}, Sansone \cite{sansone} or Thangavelu \cite{thangavelu}. In this paper we use the definition of Hermite polynomials such that they are the Gram-Schmidt orthonormalization of the polynomials $1, x, x^2, x^3,\ldots$ with respect to the standard Gaussian measure.
This corresponds to the definition used in \cite{bogachev}.

\begin{definition}\label{def:hermite}
For any $k\in\N_0$ we call
\begin{align*}
H_{k}(x)=\frac{(-1)^k}{\sqrt{k!}}\, e^{\frac{x^2}{2}}\frac{d^k}{dx^k}e^{-\frac{x^2}{2}}
\end{align*}
the \emph{$k$-th (univariate) Hermite polynomial} and for any multi-index $\k\in\N_0^d$ the \emph{(multivariate) Hermite polynomial} $H_{\k}(\x)$ with $\x\in\R^d$ is defined by
\begin{align*}
H_{\k}(\x)=\prod_{j=1}^{d}H_{k_j}(x_j).
\end{align*}
\end{definition}

The exponential generating function $G$ of the Hermite polynomials is given by 
\begin{align}\label{eq:generating}
G(\x,\t):=e^{\x\tr\t-\frac{\t\tr\t}{2}}=\sum_{\k\in\N_0^d}H_\k(\x)\frac{\t^\k}{\sqrt{\k!}}
\end{align}
with $\x,\t\in\R^d$, see Bogachev \cite[Chapter 1.3]{bogachev} for the case of $d=1$. The exponential generating function and the representation of the Hermite polynomials via function $G$, i.e., $H_\k(\x)=\sqrt{\k!}\frac{d^\k}{d\t^\k}G(\x,\t)\vert_{\t=\0}$, is a useful tool to work with these kind of polynomials.

\begin{proposition}\label{prop:ONB}
The sequence of Hermite polynomials $\left(H_\k(\x)\right)_{\k\in\N_0^d}$ forms an orthonormal basis of the function space $L^2(\R^d,\varphi_d)$. 
\end{proposition}
\begin{proof}
A proof can be found in Bogachev \cite{bogachev}, Lemma 1.3.2. and Corollary 1.3.3.
\end{proof}

Let $e_i=(0,\ldots,0,1,0,\ldots,0)\in\N_0^d$ with the $1$ at the $i$-th entry. Then for $\k,\ellb\in\N_0^d$ and for $\x,\y \in\R^d$ we have the following properties:
\begin{enumerate}[(i)]
	\item\label{lem:hermite:cramer} $\left\vert H_\k(\x)\varphi_d(\x)^{\frac{1}{2}}\right\vert\leq\mu^d \left(2\pi\right)^{-\frac{d}{4}} \leq 1$ with constant $\mu$ smaller than $1.086435$ (Cramer's bound);
	\item\label{lem:hermite:rec} $H_{\k+e_i}(\x)=\frac{1}{\sqrt{k_{i}+1}}\left(x_iH_{\k}(\x)-\frac{\partial}{\partial x_i}H_{\k}(\x)\right)$;
	\item\label{lem:hermite:deriv} $\frac{\partial^{\vert \ellb\vert}}{\partial \x^\ellb}H_{\k}(\x)=\begin{cases}\sqrt{\frac{\k!}{(\k-\ellb)!}} H_{\k-\ellb}(\x) & \textnormal{if } \k\geq\ellb;\\ 0 & \textnormal{otherwise.}\end{cases}$
\end{enumerate}
This formulation of Cramer's bound follows from the slightly different statement for an alternative definition of the Hermite polynomials in Sansone \cite{sansone} Section 4.5, where the author refers to Cramer \cite{cramer}. The properties \eqref{lem:hermite:rec} and \eqref{lem:hermite:deriv} are easily deduced from the Definition \ref{def:hermite} and \eqref{eq:generating}. Next we define a first order differential operator $\cD_{\!x}$ as
\begin{align}\label{eq:diffop}
\cD_{\!x}:=\frac{\partial}{\partial x}-x,
\end{align}
which is motivated by the above Property \eqref{lem:hermite:rec}. Moreover, we have that $-\cD_{\!x_i}$ is the adjoint operator of $\frac{\partial}{\partial x_i}$ with respect to the $L^2(\R^d,\varphi_d)$ inner product,
\begin{align*}
\int_{\R^d}\frac{\partial}{\partial x_i}\big(H_\k(\x)\big)H_\ellb(\x)\varphi_d(\x)d\x&=-\int_{\R^d}H_\k(\x)\frac{\partial}{\partial x_i}\left(H_\ellb(\x)\varphi_d(\x)\right)d\x\\
&=-\int_{\R^d}H_\k(\x)\left(\frac{\partial}{\partial x_i}H_\ellb(\x)-x_iH_\ellb(\x)\right)\varphi_d(\x)d\x\\
&=-\int_{\R^d}H_\k(\x)\cD_{\!x_i}\big(H_\ellb(\x)\big)\varphi_d(\x)d\x.
\end{align*}
Recall that we have for the differential operator $\cD=(\cD_{\!x_1},\ldots,\cD_{\!x_d})$ the multi-index notation $\cD^\ellb=\cD_{\!x_1}^{\ell_1}\cdots \cD_{\!x_d}^{\ell_d}$ for any $\ellb\in\N_0^d$.

\subsection{Hermite expansion}

Since we know from Proposition \ref{prop:ONB} that the Hermite polynomials form an orthonormal basis of $L^2(\R^d,\varphi_d)$, we can write any function $f\in L^2(\R^d,\varphi_d)$ as a Hermite series, i.e.,
\begin{align}\label{eq:hermiteexpansion}
f(\x)=\sum_{\k\in\N_0^d}\hat{f}(\k)H_\k(\x)
\end{align}
where the sum converges in the $L^2(\R^d,\varphi_d)$-norm and $\hat{f}(\k)$ is the $\k$-th Hermite coefficient given by
\begin{align*}
\hat{f}(\k)=\int_{\R^d}f(\x)H_\k(\x)\varphi_d(\x)d\x.
\end{align*}
Note that the $\k$-th Hermite coefficient exists for every $f\in L^2(\R^d,\varphi_d)$, because
\begin{align*}
\int_{\R^d}|f(\x)H_\k(\x)\varphi_d(\x)|d\x\leq \left(\int_{\R^d}f(\x)^2\varphi_d(\x)d\x\right)^{1/2}\left(\int_{\R^d}H_\k(\x)^2\varphi_d(\x)d\x\right)^{1/2}<\infty.
\end{align*}
Furthermore we can show that the Hermite expansion is unique under the assumption of continuity.
\begin{proposition}\label{prop:unique}
Let $f\in L^2(\R^d,\varphi_d)$ be a continuous function. If $\hat{f}(\k)=0$ for all $\k\in\R^d$, then $f\equiv 0$.
\end{proposition}
\begin{proof}
We know that $f(\x)=\sum_{\k\in\N_0^d}\hat{f}(\k)H_\k(\x)$ holds in the $L^2(\R^d,\varphi_d)$ sense and thus, using the Parseval identity
\begin{align*}
\int_{\R^d}f(\x)^2\varphi_d(\x)d\x=\sum_{\k\in\N_0^d}\hat{f}(\k)^2=0.
\end{align*}
Since $f$ is continuous, we get that $f\equiv 0$.
\end{proof}

\begin{corollary}
Let $f,g\in L^2(\R^d,\varphi_d)$ be continuous functions with the same Hermite coefficients. Then $f=g$.
\end{corollary}
\begin{proof}
The statement follows by applying Proposition \ref{prop:unique} to $f-g$.
\end{proof}

In general, the Hermite expansion \eqref{eq:hermiteexpansion} of an $L^2(\R^d,\varphi_d)$-function does not need to converge for fixed $\x\in\R^d$, but under further assumptions on the function as well as on its Hermite coefficients the Hermite expansion converges even pointwise.

\begin{proposition}\label{prop:hermiteexpansion}
Let $f\in L^2(\R^d,\varphi_d)$ be continuous with $\sum_{\k\in\N_0^d}|\hat{f}(\k)|<\infty$. Then
\begin{align*}
f(\x)=\sum_{\k\in\N_0^d}\hat{f}(\k)H_\k(\x)
\end{align*}
for all $\x\in\R^d$.
\end{proposition}
\begin{proof}
Because of Cramer's bound, see Property \eqref{lem:hermite:cramer} of the Hermite polynomials, we get
\begin{align*}
\vert\hat{f}(\k)H_\k(\x)\varphi_d(\x)^{\frac{1}{2}}\vert\leq \vert\hat{f}(\k)\vert
\end{align*}
for each $\k\in\N_0^d$. Since $\sum_{\k\in\N_0^d}|\hat{f}(\k)|<\infty$, Weierstrass' uniform convergence theorem, see e.g.\ Rudin \cite{rudin76} Theorem 7.10, states that $\sum_{k\in\N_0^d}\hat{f}(\k)H_\k(\x)\varphi_d(\x)^{\frac{1}{2}}$ converges uniformly towards a function $\bar{f}$. Furthermore, due to the uniform limit theorem, see e.g.\ Rudin \cite{rudin76} Theorem 7.12, we have that $\bar{f}$ is continuous, because $H_\k(\x)\varphi_d(\x)^{\frac{1}{2}}$ is continuous for each $\k\in\N_0^d$. Thus, we even get that
\begin{align*}
\sum_{\k\in\N_0^d}\hat{f}(\k)H_\k(\x)=\bar{f}(\x)\varphi_d(\x)^{-\frac{1}{2}}
\end{align*}
holds for all $\x\in\R^d$. Since $\x\mapsto\bar{f}(\x)\varphi_d(\x)^{-\frac{1}{2}}$ has the same Hermite coefficients as $f$ and both functions are continuous, we know from the uniqueness of the Hermite expansion that $f(\x)=\bar{f}(\x)\varphi_d(\x)^{-\frac{1}{2}}$ for all $\x\in\R^d$. So we end up with 
\begin{align*}
\sum_{\k\in\N_0^d}\hat{f}(\k)H_\k(\x)=f(\x),
\end{align*}
which holds for all $\x\in\R^d$.
\end{proof}

\section{Quasi-Monte Carlo integration in weighted Hermite spaces}\label{sec:integration}

In this section we study quasi-Monte Carlo integration with respect to the Gaussian measure. For that, we set up a reproducing kernel Hilbert space of functions on the $\R^d$ and we further show that under additional assumptions QMC integration in these spaces is polynomially tractable. 

\subsection{Hermite spaces}\label{ssec:hermitespace}

\begin{definition}
Let $\r:\N_0^d\longrightarrow\R^{+}$ be a summable function, i.e., $\sum_{\k\in\N_0^d}\r(\k)<\infty$. For $f\in L^2(\R^d,\varphi_d)$ let
\begin{align*}
\|f\|_\r:=\Bigg(\sum_{\k\in\N_0^d}\r(\k)^{-1}\left\vert\hat{f}(\k)\right\vert^2\Bigg)^{\frac{1}{2}}
\end{align*}
where $\hat{f}(\k)$ is the $\k$-th Hermite coefficient of $f$. Then we call
\begin{align*}
\cH_\r:=\left\{f\in \cL^2(\R^d,\varphi_d)\cap C(\R^d)\,:\, \|f\|_\r<\infty\right\}
\end{align*}
a \emph{Hermite space}.
\end{definition}

Note that $\|\cdot\|_\r$ is a semi-norm on $\cL^2(\R^d,\varphi_d)$, but under the additional assumption, that the functions are continuous as well, it is a norm. Thus, a Hermite space is a Banach space with norm $\|\cdot\|_\r$; in fact it is a Hilbert space when equipped with the inner product
\begin{align*}
\langle f,g\rangle_\r&:=\frac{1}{4}\left(\|f+g\|_\r^2-\|f-g\|_\r^2\right)\\
&=\sum_{\k\in\N_{0}^{d}}\r(\k)^{-1}\hat{f}(\k)\hat{g}(\k).
\end{align*}

\begin{theorem}\label{th:hermitespaceexpansion}
Let $f\in\cH_\r$. Then, the Hermite expansion of $f$ converges pointwise, i.e.,
\begin{align*}
f(\x)=\sum_{\k\in\N_0^d}\hat{f}(\k)H_\k(\x)
\end{align*}
holds for all $\x\in\R^d$.
\end{theorem}
\begin{proof}
We have that $f$ is a Gaussian square-integrable and continuous function. Moreover, we get from the Cauchy-Schwarz inequality 
\begin{align*}
\sum_{\k\in\N_0^d}\vert\hat{f}(\k)\vert&=\sum_{\k\in\N_0^d}\vert\r(\k)^{\frac{1}{2}}\r(\k)^{-\frac{1}{2}}\hat{f}(\k)\vert\\
&\leq \left(\sum_{\k\in\N_0^d}\r(\k)\right)^{\frac{1}{2}}\left(\sum_{\k\in\N_0^d}\r(\k)^{-1}\hat{f}(\k)^2\right)^{\frac{1}{2}}\\
&=\left(\sum_{\k\in\N_0^d}\r(\k)\right)^{\frac{1}{2}}\|f\|_{\r}<\infty,
\end{align*}
because $\r$ is summable. The statement now follows from Proposition \ref{prop:hermiteexpansion}.
\end{proof}

We shall show shortly that a Hermite space is always a reproducing kernel Hilbert space. For that, we first recall the definition of a reproducing kernel Hilbert space as well as some important properties.

\begin{definition}
A Hilbert space $\cH$ of functions $f:X\longrightarrow \R$ on a set $X$ with inner product $\langle\cdot,\cdot\rangle$ is called a reproducing kernel Hilbert space, if there exists a function $K:X\times X\longrightarrow\R$ such that
\begin{enumerate}[(i)]
	\item\label{it:kernel} $K(\cdot,y)\in\cH$ for each fixed $y\in X$ and
	\item\label{it:repro} $\langle f,K(\cdot,y)\rangle=f(y)$ for each fixed $y\in X$ and for all $f\in \cH$ (reproducing property).
\end{enumerate}
\end{definition}
In this definition we use the notation $K(\cdot,y)$ to indicate that $K$ is a function of the first variable, i.e., $x\mapsto K(x,y)$, and the inner product $\langle f,K(\cdot,y)\rangle$ is taken with respect to the first variable of $K$. The function $K$ is called reproducing kernel and is unique: Assume there is another function $\tilde{K}(\cdot,y)\in\cH$ satisfying the reproducing property. Then, with the reproducing property of $K$ and $\tilde{K}$ we get
\begin{align*}
\tilde{K}(x,y)=\langle \tilde{K}(\cdot,y),K(\cdot,x)\rangle=\langle K(\cdot,x),\tilde{K}(\cdot,y)\rangle=K(y,x)=K(x,y)
\end{align*}
where the last equality holds because of the symmetry of the reproducing kernel,
\begin{align*}
K(x,y)=\langle K(\cdot,y),K(\cdot,x)\rangle=\langle K(\cdot,x),K(\cdot,y)\rangle=K(y,x).
\end{align*}
More information about the theory of reproducing kernel Hilbert spaces can be found in Aronszajn \cite{aronszajn50}.

\begin{theorem}\label{th:RKHS}
A Hermite space $\cH_\r$ is a reproducing kernel Hilbert space and the reproducing kernel function $K_\r:\R^d\times\R^d$ is given by
\begin{align}\label{eq:hermitekernel}
K_\r(\x,\y)=\sum_{\k\in\N_0^d}\r(\k)H_\k(\x)H_\k(\y).
\end{align}
\end{theorem}
\begin{proof}
First, we show that $K_\r(\cdot,\y)$ belongs to the Hermite space $\cH_\r$ for each $\y\in\R^d$. The $\k$-th Hermite coefficient of $K_\r$ as a function in $\x$ is
\begin{align*}
\widehat{K_\r(\cdot,\y)}(\k)&=\int_{\R^d}K_\r(\x,\y)H_\k(\x)\varphi_d(\x)d\x\\
&=\sum_{\j\in\N_0^d}\r(\j)\int_{\R^d}H_\j(\x)H_\k(\x)\varphi_d(\x)d\x\,H_\j(\y)\\
&=\r(\k)H_\k(\y).
\end{align*}
So we get that
\begin{align*}
\|K_\r(\cdot,\y)\|_\r^2&=\sum_{\k\in\N_0^d}\r(\k)^{-1}\left(\widehat{K_\r(\cdot,\y)}(\k)\right)^2\\
&=\sum_{\k\in\N_0^d}\r(\k)^{-1}\r(\k)^2H_\k(\y)^2\\
&\leq c e^{\frac{\y\tr\y}{2}}\sum_{\k\in\N_0^d}\r(\k)<\infty,
\end{align*}
where Cramer's bound is used to estimate $H_\k(\y)$ by $c e^{\frac{\y\tr\y}{2}}$ with positive constant $c$. Next we have to verify that $K_\r$ satisfies the reproducing property. For any $f\in\cH_\r$ we have that
\begin{align*}
\langle f, K_\r(\cdot,\y)\rangle_\r &= \sum_{\k\in\N_0^d}\r(\k)^{-1}\hat{f}(\k)\widehat{K_\r(\cdot,\y)}(\k)\\
&=\sum_{\k\in\N_0^d}\r(\k)^{-1}\hat{f}(\k)\r(\k)H_\k(\y)\\
&=\sum_{\k\in\N_0^d}\hat{f}(\k)H_\k(\y)\\
&=f(\y)
\end{align*}
holds. Thus, the function $K_\r$, given by \eqref{eq:hermitekernel}, is indeed the reproducing kernel of the Hermite space $\cH_\r$ and consequently, $\cH_\r$ is a reproducing kernel Hilbert space.
\end{proof}

It is well-known that function evaluation in a reproducing kernel Hilbert space is a continuous linear functional on that space. 

We give two examples of Hermite spaces where the function $\r$ controls both the decay of the Hermite coefficients of the functions and the influence of the coordinates of the variables.

\subsubsection{Weighted Hermite spaces with polynomially decaying coefficients}

Let $\alpha>1$ and $\gamma>0$. Then we define a function $p_{\alpha,\gamma}:\N_0\longrightarrow\R^{+}$ as
\begin{align*}
p_{\alpha,\gamma}(k):=\begin{cases}1 & \textnormal{if } k=0,\\ \gamma k^{-\alpha} & \textnormal{if } k\neq 0\end{cases}
\end{align*}
and we note that
\begin{align*}
\sum_{k=0}^{\infty}p_{\alpha,\gamma}(k)=1+\gamma\sum_{k=1}^{\infty}k^{-\alpha}=1+\gamma\zeta(\alpha)<\infty
\end{align*}
where $\zeta$ is the Riemann zeta function. Since $p_{\alpha,\gamma}$ is summable, we can set $\r=p_{\alpha,\gamma}$ and consider the Hermite space $\cH_{p_{\alpha,\gamma}}$ of functions on the $\R$. Then, the norm of $\cH_{p_{\alpha,\gamma}}$ can be written as
\begin{align*}
\|f\|_{p_{\alpha,\gamma}}^2=1+\gamma^{-1}\sum_{k=1}^{\infty}k^{\alpha}\hat{f}(k)^2.
\end{align*}
Hence, the Hermite coefficients of $f$ have to decay polynomially with a rate of $o(k^{-\alpha/2})$ so that the norm becomes finite and, consequently, the function $f$ belongs to the Hermite space $\cH_{p_{\alpha,\gamma}}$. According to Theorem \ref{th:RKHS} the reproducing kernel of $\cH_{p_{\alpha,\gamma}}$ is of the form
\begin{align*}
K_{p_{\alpha,\gamma}}(x,y)=1+\gamma\sum_{k=1}^{\infty}k^{-\alpha}H_k(x)H_k(y).
\end{align*}

The next theorem states sufficient conditions for a function to be in the Hermite space $\cH_{p_{\alpha,\gamma}}$ of functions on the $\R$. These conditions concern both the smoothness and the asymptotic behavior of the function at infinity. The theorem further gives a relation between the ``smoothness parameter'' $\alpha$ of the Hermite space and the actual smoothness of the function.

\begin{theorem}\label{th:smoothness}
Let $\beta>2$ be an integer and $f:\R\longrightarrow\R$ be a $\beta$-times differentiable function satisfying
\begin{enumerate}[(i)]
\item $\cD_x^jf(x)\varphi(x)^{\frac{1}{2}}$ is Lebesgue integrable for each $j\in\{1,\ldots,\beta\}$ and
\item\label{infbehavior} $\cD_x^jf(x)=O\big(e^{x^2/(2c)}\big)$ as $|x|\rightarrow\infty$ for each $j\in\{0,\ldots,\beta-1\}$ and some $c>1$,
\end{enumerate}
where $\cD_x$ is the differential operator defined by \eqref{eq:diffop}. Then $f\in\cH_{p_{\alpha,\gamma}}$ for any  $\alpha\in(1,\beta-1)$ and any $\gamma>0$.
\end{theorem}
\begin{proof}
Using integration by parts we obtain
\begin{align*}
\hat{f}(k)&=\int_{\R}f(x)H_k(x)\varphi(x)dx\\
&=\lim_{a\to\infty}f(x)\varphi(x)\frac{H_{k+1}(x)}{\sqrt{k+1}}\Big\vert_{x=-a}^{a}-\int_{\R}\frac{d}{dx}\left(f(x)\varphi(x)\right)\frac{H_{k+1}(x)}{\sqrt{k+1}}dx.
\end{align*}
Because of assumption \eqref{infbehavior} there is a constant $\hat{c}>0$ such that
\begin{align*}
|f(a)\varphi(a)H_{k+1}(a)|&\leq \hat{c}\,e^{(\frac{1}{c}-1)\frac{a^2}{2}}\left|H_{k+1}(a)\right|
\end{align*}
and, since $\frac{1}{c}-1<0$ and $H_{k+1}(a)$ is a polynomial, we get that
\begin{align*}
\lim_{|a|\to\infty}f(a)\varphi(a)H_{k+1}(a)=0.
\end{align*}
Thus,
\begin{align*}
\hat{f}(k)&=-\frac{1}{\sqrt{k+1}}\int_{\R}\left(\frac{d}{dx}f(x)-xf(x)\right)H_{k+1}(x)\varphi(x)dx\\
&=-\frac{1}{\sqrt{k+1}}\int_{\R}\cD_xf(x)H_{k+1}(x)\varphi(x)dx\\
&=-\frac{1}{\sqrt{k+1}}\,\widehat{\cD_xf}(k+1).
\end{align*}
Now we can proceed in the same way, but for $\cD^j_xf$, $j=1,\ldots,\beta-1$, instead of $f$, such that after $\beta-1$ times we end up with
\begin{align*}
\hat{f}(k)=(-1)^\beta \bigg(\frac{k!}{(k+\beta)!}\bigg)^{\frac{1}{2}}\,\widehat{\cD_x^{\beta}f}(k+\beta).
\end{align*}
So the $k$-th Hermite coefficient of $f$ can be estimated by
\begin{align*}
\vert\hat{f}(k)\vert&\leq \left(\frac{k!}{(k+\beta)!}\right)^{\frac{1}{2}} \int_{\R}\big|\cD^{\beta}f(x)\big|\left|H_{k+\beta}(x)\right|\varphi(x)dx\\
&\leq \left(\frac{k!}{(k+\beta)!}\right)^{\frac{1}{2}} \int_{\R}\big|\cD^{\beta}f(x)\varphi(x)^{\frac{1}{2}}\big|dx\\
&= C\,\left(\frac{k!}{(k+\beta)!}\right)^{\frac{1}{2}}
\end{align*}
where $C=\int_{\R}|\cD_x^\beta f(x)\varphi(x)^{\frac{1}{2}}|dx<\infty$ and  Cramer's bound is applied for obtaining the second estimate. Then the norm of $f$ can be bounded,    
\begin{align*}
\|f\|_{p_{\alpha,\gamma}}^2&=\sum_{k\in\N_0}p_{\alpha,\gamma}(k)|\hat{f}(k)|^2\\
&\leq |\hat{f}(0)|^2+\gamma^{-1}C^2\sum_{k=1}^{\infty}k^\alpha \frac{k!}{(k+\beta)!}\\
&\leq |\hat{f}(0)|^2+\gamma^{-1}C^2\sum_{k=1}^{\infty}k^{-\beta+\alpha}\\
&= |\hat{f}(0)|^2+\gamma^{-1}C^2\,\zeta(\beta-\alpha)<\infty,
\end{align*}
if $\beta-\alpha>1$. Therefore, we have that $f$ is in the Hermite space $\cH_{p_{\alpha,\gamma}}$ for any $\alpha\in(1,\beta-1)$.
\end{proof}

We can also show a reverse statement, i.e., the space $\cH_{p_{\alpha,\gamma}}$ contains functions which are differentiable up to some finite order. 

\begin{theorem}\label{prop:differentiability}
Let $f\in\cH_{p_{\alpha,\gamma}}$. Then $f\in C^{\beta}(\R^d)$ for every $\beta<\alpha-1$.
\end{theorem}
\begin{proof}
For $n\in\N$ we write $f_n(x)=\sum_{k=0}^{n}\widehat{f}(k)H_k(x)$ and let $\ell\in\N$ with $\ell\leq\beta$. We have that
\begin{align*}
\frac{d^{\ell}}{d x^\ell}f_n(x)=\sum_{k=0}^{n}\widehat{f}(k)\frac{d^{\ell}}{d x^\ell}H_k(x)=\sum_{k=\ell}^{n}\widehat{f}(k)\sqrt{\frac{k!}{(k-\ell)!}}\,H_{k-\ell}(x).
\end{align*}
Then, for any $n\in\N$,
\begin{align*}
\left\vert \frac{d^{\ell}}{d x^\ell}f_n(x)\right\vert&=\left\vert \sum_{k=\ell}^{n}\left(\widehat{f}(k)p_{\alpha,\gamma}(k)^{-\frac{1}{2}}\right)\left(p_{\alpha,\gamma}(k)^{\frac{1}{2}}\,\sqrt{\frac{k!}{(k-\ell)!}}H_{k-\ell}(x)\right)\right\vert\\
&\leq \|f\|_{p_{\alpha,\gamma}}\frac{1}{\varphi(x)}\left(\sum_{k=\ell}^{n}p_{\alpha,\gamma}(k)\frac{k!}{(k-\ell)!}\right)^\frac{1}{2}\\
&\leq \|f\|_{p_{\alpha,\gamma}}\frac{\gamma}{\varphi(x)}\left(\sum_{k=\ell}^{n}k^{-\alpha}\frac{k!}{(k-\ell)!}\right)^\frac{1}{2}\\
\end{align*}
Since $k!/(k-\ell)!\leq k^{\ell}$, we get
\begin{align*}
\left\vert \frac{d^{\ell}}{d x^\ell}f_n(x)\right\vert&\leq\|f\|_{p_{\alpha,\gamma}}\frac{\gamma}{\varphi(x)}\left(\sum_{k=1}^{n}k^{-(\alpha-\ell)}\right)^\frac{1}{2}.
\end{align*}
Moreover, we have that
\begin{align*}
\sum_{k=1}^{\infty}k^{-(\alpha-\ell)}=\zeta(\alpha-\ell)<\infty,
\end{align*}
because $\alpha-\ell\geq\alpha-\beta>\alpha-(\alpha-1)>1$. Thus, we know that the sequence $(\frac{d^{\ell}}{d x^\ell}f_n(x))_{n\in\N}$ converges uniformly on compacts. Furthermore, from Theorem \ref{th:hermitespaceexpansion} it follows that $f_n$ converges pointwise to $f$. Because of a basic theorem about uniform convergence and differentiation, see e.g., Rudin \cite[Theorem 7.17]{rudin76} we get that $\frac{d^\ell}{dx^\ell}f$ exists and it is given by
\begin{align*}
\frac{d^\ell}{dx^\ell}f(x)=\sum_{k=0}^{\infty}\widehat{f}(k)\sqrt{\frac{k!}{(k-\ell)!}}H_{k-\ell}(x).
\end{align*}
The continuity of the derivatives also follows by a basic theorem about uniform convergence, see e.g., Rudin \cite[Theorem 7.12]{rudin76}.
\end{proof}

Now, we suggest to generalize this type of Hermite space to functions on the $\R^d$ in the following manner. Let $\gammab=(\gamma_1,\ldots,\gamma_d)$ with $\gamma_j>0$ be non-increasing weights and $\alphab=(\alpha_1,\ldots,\alpha_d)$ with $\alpha_j>1$ smoothness parameters. Then we define the function $\p_{\alphab,\gammab}:\N_0^d\longrightarrow\R^+$ by
\begin{align}\label{def:poly}
\p_{\alphab,\gammab}(\k):=\prod_{j=1}^{d}p_{\alpha_j,\gamma_j}(k_j),
\end{align}
which is summable, because
\begin{align}\label{eq:polysum}
\sum_{\k\in\N_0^d}\p_{\alphab,\gammab}(\k)=\prod_{j=1}^{d}(1+\gamma_j\zeta(\alpha_j))<\infty.
\end{align}
So we get the Hermite space $\cH_{\p_{\alphab,\gammab}}$ which also can be written as the $d$-fold Hilbert space tensor product of Hermite spaces of functions on the $\R$,
\begin{align*}
\cH_{\p_{\alphab,\gammab}}=\cH_{p_{\alpha_1,\gamma_1}}\otimes\cdots\otimes\cH_{p_{\alpha_d,\gamma_d}}.
\end{align*}
The reproducing kernel $K_{\p_{\alphab,\gammab}}$ is the product of 
the kernels $K_{p_{\alpha_i,\gamma_j}}$,
\begin{align*}
K_{\p_{\alphab,\gammab}}(\x,\y)&=\sum_{\k\in\N_0^d}\p_{\alphab,\gammab}(\k)H_\k(\x)H_\k(\y)\\
&=\prod_{j=1}^{\infty}\left(1+\gamma_j\sum_{k=1}^{\infty}k^{-\alpha_j}H_{k}(x_j)H_{k}(y_j)\right)
\end{align*}
for $\x,\y\in\R^d$. The Hermite coefficients of the functions of $\cH_{\p_{\alphab,\gammab}}$ again have to decay polynomially and with the smoothness parameters $\alphab$ one can control the order of the polynomial decay for each direction separately. The weights $\gammab$ determine the influence of the coordinates. Since we assumed that $\gammab$ is a non-increasing sequence, the weights moderate the influence of the coordinates. That means that each coordinate is less or at most equally important than the previous coordinates. Because of the product-form of the function $\p_{\alphab,\gammab}$ we are dealing with so-called product weights, first introduced in Sloan and Wo\'zniakowski \cite{sw98}. Because of the weights $\gammab$ we also call $\cH_{\p_{\alphab,\gammab}}$ a \emph{weighted Hermite space}. Furthermore, note that Theorem \ref{th:smoothness} can be extended to Hermite spaces $\cH_{\p_{\alphab,\gammab}}$ of functions on the $\R^d$ which means that the smoothness-parameters $\alphab$ correspond to the differentiability of the functions.

\subsubsection{Weighted Hermite spaces with exponentially decaying coefficients}

Let $\gammab=(\gamma_1,\ldots, \gamma_d)$ be a sequence of non-increasing positive weights and $\omegab=(\omega_1,\ldots,\omega_d)\in(0,1)^d$. We define a function $\epsilonb_{\omegab,\gammab}:\N_0^d\longrightarrow\R^+$ as
\begin{align}\label{def:exp}
\epsilonb_{\omegab,\gammab}(\k)=\prod_{j=1}^{d}\epsilon_{\omega_j,\gamma_j}(k_j)
\end{align}
with
\begin{align*}
\epsilon_{\omega_j,\gamma_j}(k_j)=\begin{cases}1 & \textnormal{if } k_j=0,\\ \gamma_j\omega_j^{k_j} & \textnormal{if } k_j\neq 0\end{cases}
\end{align*}
for $j=1,\ldots,d$. Since
\begin{align*}
\sum_{\k\in\N_0^d}\epsilonb_{\omegab,\gammab}(\k)&=\prod_{j=1}^{d}\left(1+\gamma_j\sum_{k=1}^{\infty}\omega_j^k\right)\\
&=\prod_{j=1}^{d}\left(1+\gamma_j\frac{\omega_j}{1-\omega_j}\right)<\infty,
\end{align*}
we have that $\epsilonb_{\omegab,\gammab}$ is summable and we get the Hermite space $\cH_{\epsilonb_{\omegab,\gammab}}$ which has the Hilbert space tensor product form
\begin{align*}
\cH_{\epsilonb_{\omegab,\gammab}}=\cH_{\epsilon_{\omega_1,\gamma_1}}\otimes\cdots\otimes\cH_{\epsilon_{\omega_d,\gamma_d}}.
\end{align*}
The Hermite coefficients of the elements of $\cH_{\epsilonb_{\omegab,\gammab}}$ decay exponentially which is controlled by the parameter $\omegab$. We again have product weights and the influence of the coordinates can be determined by the choice of the weights $(\gamma_1,\ldots,\gamma_d)$. 

The reproducing kernel of $\cH_{\epsilonb_{\omegab,\gammab}}$ is given by
\begin{align*}
K_{\epsilonb_{\omegab,\gammab}}(\x,\y)&=\prod_{j=1}^{d}\left(1+\gamma_j\sum_{k=1}^{\infty}\omega_j^{k}H_k(x_j)H_k(y_j)\right),
\end{align*}
which follows from Theorem \ref{th:RKHS}. From Mehler's formula, see Szeg\H{o} \cite{szego}, we can further derive, for each $j=1,\ldots,d$, the following formula
\begin{align*}
1+\gamma_j\sum_{k=1}^{\infty}\omega_j^{k}H_k(x_j)H_k(y_j)=1-\gamma_j+\gamma_j\frac{1}{\sqrt{1-\omega_j^2}}\exp\left(\frac{\omega_j}{1+\omega_j}x_j y_j-\frac{\omega_j^2}{2(1-\omega_j^2)}(x_j-y_j)^2\right)
\end{align*}
and thus,
\begin{align}\label{eq:kernelexp}
K_{\epsilonb_{\omegab,\gammab}}(\x,\y)&=\prod_{j=1}^{d}\left(1-\gamma_j+\gamma_j\frac{1}{\sqrt{1-\omega_j^2}}\exp\left(\frac{\omega_j}{1+\omega_j}x_j y_j-\frac{\omega_j^2}{2(1-\omega_j^2)}(x_j-y_j)^2\right)\right).
\end{align}
From this representation of the reproducing kernel we see that for $\gamma_j<1$, $j=1,\ldots,d$, the reproducing kernel $K_{\epsilonb_{\omegab,\gammab}}$ is positive.  

We want to know more about the functions that are contained in the Hermite space $\cH_{\epsilon_{\omegab,\gammab}}$. In the previous subsection, Theorem \ref{th:smoothness} gave us sufficient conditions on a function such that the function belongs to the Hermite space of functions with polynomially decaying Hermite coefficients. Unfortunately, we do not have a similar statement concerning the Hermite space of exponential decaying Hermite coefficients. Nevertheless, we know that the Hermite space $\cH_{\epsilon_{\omegab,\gammab}}$ contains all polynomials, because $\lin\{H_\k:\vert\k\vert\leq m\}$ is equal to the set of all polynomials on the $\R^d$ with degree smaller or equal to $m$. Furthermore, the function $f:\R^d\rightarrow \R$, given by $f(\x)=f(x_1,\ldots,x_d)=\exp(\w\tr \x)$ with $\w\in\R^d$, belongs to $\cH_{\epsilon_{\omegab,\gammab}}$ for any weight sequence $\gammab$ and any $\omegab$. This is true, because we have that the $\k$-th Hermite coefficient is given by
\begin{align*}
\widehat{f}(\k)=e^{\w\tr\w}\frac{\w^\k}{\sqrt{\k!}}
\end{align*}
and consequently, the norm of $f$ is finite, 
\begin{align*}
\|f\|_{\epsilon_{\omegab,\gammab}}^2=e^{\w\tr\w}\prod_{j=1}^{d}\left(1+\gamma_j\sum_{k=1}^{\infty}\omega_j^k\frac{w_j^{2k}}{k!}\right)\leq e^{\w\tr\w}\prod_{j=1}^{d}\exp(\omega_j w_j^{2})<\infty.
\end{align*}
In addition, we can show that the functions in $\cH_{\epsilon_{\omegab,\gammab}}$ are analytic, i.e., analyticity is a necessary condition on the functions for being in the Hermite space of exponentially decaying Hermite coefficients. 

\begin{proposition}\label{prop:analytic}
Let $f\in\cH_{\epsilon_{\omegab,\gammab}}$. Then $f$ is analytic.
\end{proposition}
\begin{proof}
Since $\gamma_j\leq 1$ for all $j=1,\ldots,d$,  we have $\cH_{\epsilon_{\omegab,\gammab}}\subseteq\cH_{\epsilon_{\omegab,\1}}$ with $\1=(1,\ldots,1)$. Therefore, it is sufficient to show analyticity for functions $f$ which belong to $\cH_{\epsilon_{\omegab,\1}}$. Let $\ellb\in\N_0^d$ and $f\in\cH_{\epsilon_{\omegab,\1}}$. Analogue to the proof of Proposition \ref{prop:differentiability} it can be shown that the derivative $\frac{\partial^{|\ellb|}}{\partial \x^\ellb}f$ exists and that we can interchange summation and differentiation, i.e., we obtain
\begin{align*}
\frac{\partial^{|\ellb|}}{\partial \x^\ellb}f(\x)=\sum_{\k\in\N_0^d}\widehat{f}(\k)\frac{\partial^{|\ellb|}}{\partial \x^\ellb}H_\k(\x)=\sum_{\k\geq\ellb}\widehat{f}(\k)\sqrt{\frac{\k!}{(\k-\ellb)!}}H_{\k-\ellb}(\x).
\end{align*}
Then,
\begin{align*}
\left\vert \frac{\partial^{|\ellb|}}{\partial \x^\ellb}f(\x)\right\vert&=\left\vert \sum_{\k\geq\ellb}\left(\widehat{f}(\k)\left(\epsilon_{\omegab,\1}(\k)\right)^{-\frac{1}{2}}\right)\left(\left(\epsilon_{\omegab,\1}(\k)\right)^{\frac{1}{2}}\sqrt{\frac{\k!}{(\k-\ellb)!}}H_{\k-\ellb}(\x)\right)\right\vert\\
&\leq \|f\|_{\epsilon_{\omegab,\1}}\frac{1}{\varphi_d(\x)}\left(\sum_{\k\geq\ellb}\epsilon_{\omegab,\1}(\k)\frac{\k!}{(\k-\ellb)!}\right)^\frac{1}{2}\\
&=\|f\|_{\epsilon_{\omegab,\1}}\frac{1}{\varphi_d(\x)}\left(\sum_{\k\geq\ellb}\prod_{j=1}^{d}(\ell_j!)\omega_j^{k_j}\frac{k_j!}{(k_j-\ell_j)!(\ell_j!)}\right)^\frac{1}{2}\\
&\leq\|f\|_{\epsilon_{\omegab,\1}}\frac{1}{\varphi_d(\x)}\prod_{j=1}^{d}\sqrt{\ell_j!}\left(\sum_{k=\ell_j}^{\infty}\binom{k}{\ell_j}\omega_j^k\right)^\frac{1}{2}\\
&\leq\|f\|_{\epsilon_{\omegab,\1}}\frac{1}{\varphi_d(\x)}\sqrt{\ellb!}\prod_{j=1}^{d}\left(\frac{\omega_j^{\ell_j}}{(1-\omega_j)^{\ell_j+1}}\right)^\frac{1}{2}\\
&\leq\|f\|_{\epsilon_{\omegab,\1}}\frac{1}{\varphi_d(\x)}\sqrt{\ellb!}\prod_{j=1}^{d}\left(\frac{(\omega^*)^{\ell_j}}{(1-\omega^*)^{\ell_j+1}}\right)^\frac{1}{2}
\end{align*}
where $\omega^*=\max_j\omega_j$. Now we show that $f$ can locally be represented by its Taylor expansion. For any $\y\in\R^d$ and any $\x\in\R^d$ with $\|\x-\y\|_\infty^2<\frac{1-\omega^*}{\omega^*}$, 
\begin{align*}
&\left\vert\sum_{\ellb\in\N_0^d}\frac{1}{\ellb!}\frac{\partial^{|\ellb|}}{\partial\x^\ellb}f(\y)\prod_{j=1}^{d}(x_j-y_j)^{\ell_j}\right\vert\leq\\
&\qquad\qquad\leq \|f\|_{\epsilon_{\omegab,\1}}\frac{1}{\varphi_d(\x)}\sum_{\ellb\in\N_0^d}\prod_{j=1}^{d}\frac{1}{\sqrt{\ell_j!}}\left(\frac{(\omega^*)^{\ell_j}(x_j-y_j)^{2\ell_j}}{(1-\omega^*)^{\ell_j+1}}\right)^{\frac{1}{2}}\\
&\qquad\qquad\leq \|f\|_{\epsilon_{\omegab,\1}}\frac{1}{\varphi_d(\x)}\prod_{j=1}^{d}\sum_{\ell=0}^{\infty}\frac{1}{\sqrt{\ell!}}\left(\frac{(\omega^*)^{\ell}(x_j-y_j)^{2\ell}}{(1-\omega^*)^{\ell+1}}\right)^{\frac{1}{2}}\\
&\qquad\qquad\leq \|f\|_{\epsilon_{\omegab,\1}}\frac{1}{\varphi_d(\x)}\prod_{j=1}^{d}\left(\sum_{\ell=0}^{\infty}\frac{1}{\ell!}\right)^{\frac{1}{2}}\left(\sum_{\ell=0}^{\infty}\frac{(\omega^*)^{\ell}(x_j-y_j)^{2\ell}}{(1-\omega^*)^{\ell+1}}\right)^{\frac{1}{2}}\\
&\qquad\qquad\leq \|f\|_{\epsilon_{\omegab,\1}}\frac{1}{\varphi_d(\x)}\left(\frac{e}{1-\omega^*}\sum_{\ell=0}^{\infty}\left(\frac{\omega^*\|\x-\y\|_\infty^2}{1-\omega^*}\right)^\ell\right)^{d/2}\\
&\qquad\qquad\leq \|f\|_{\epsilon_{\omegab,\1}}\frac{1}{\varphi_d(\x)}\left(\frac{e}{1-\omega^*-\omega^*\|\x-\y\|_\infty^2}\right)^{d}<\infty.
\end{align*}
Thus, we have that the Taylor expansion converges locally. It remains to show that the remainder $R_n$ of the Taylor polynomial, given by
\begin{align*}
R_n:=\sum_{|\k|=n+1}\frac{n+1}{\k!}(\x-\y)^{\k}\int_{0}^{1}(1-t)^n \frac{\partial^{|\k|}}{\partial\x^{\k}}f(\y+t(\x-\y)) dt,
\end{align*}
vanishes if $n$ goes to infinity. We have
\begin{align*}
\vert R_n\vert&\leq\sum_{|\k|=n+1}\frac{n+1}{\k!}\vert\x-\y\vert^{\k}\int_{0}^{1}|1-t|^n \left|\frac{\partial^{|\k|}}{\partial\x^{\k}}f(\y+t(\x-\y))\right|dt\\
&\leq \sum_{|\k|=n+1}(n+1)\vert\x-\y\vert^{\k}\int_{0}^{1} \frac{\|f\|_{\epsilon_{\omegab,\1}}|1-t|^n}{\varphi_d(\y+t(\x-\y))}\prod_{j=1}^{d}\left(\frac{(\omega^{*})^{k_j}}{k_j!(1-\omega^*)^{k_j+1}}\right)^{\frac{1}{2}} dt\\
&\leq (n+1)\left[\int_{0}^{1} \frac{\|f\|_{\epsilon_{\omegab,\1}}|1-t|^n}{\varphi_d(\y+t(\x-\y))}dt \right]\sum_{|\k|=n+1}\prod_{j=1}^{d}\left(\frac{(\omega^{*})^{k_j}|x_j-y_j|^{2k_j}}{(1-\omega^*)^{k_j+1}}\right)^{\frac{1}{2}}.
\end{align*}
Since $\|\x-\y\|_\infty<\sqrt{\frac{1-\omega^*}{\omega^*}}$, we have for any $j=1,\ldots,d$, 
\begin{align*}
\frac{1}{\varphi(y_j+t(x_j-y_j))}\leq \begin{cases}1/\varphi\big(y_j+\sqrt{(1-\omega^*)/\omega^*}\,\big) & \textnormal{if } y_j\geq 0;\\1/\varphi\big(y_j-\sqrt{(1-\omega^*)/\omega^*}\,\big) & \textnormal{if } y_j< 0, \end{cases}
\end{align*}
such that we we can bound $1/\varphi_d(\y+t(\x-\y))$ by some constant $C_1$ independent of $n$ and $t$. Hence,
\begin{align*}
\vert R_n\vert&\leq C_1\|f\|_{\epsilon_{\omegab,\1}}(n+1)\left[\int_{0}^{1}|1-t|^n dt\right]\sum_{|\k|=n+1}\prod_{j=1}^{d}\left(\frac{(\omega^{*})^{k_j}|x_j-y_j|^{2k_j}}{(1-\omega^*)^{k_j+1}}\right)^{\frac{1}{2}}\\
&\leq C_1\|f\|_{\epsilon_{\omegab,\1}}\frac{1}{(1-\omega^*)^{d/2}}\sum_{|\k|=n+1}\prod_{j=1}^{d}\left(\frac{\omega^*\|\x-\y\|_\infty^2}{1-\omega^*}\right)^{\frac{k_j}{2}}\\
&\leq C_1\|f\|_{\epsilon_{\omegab,\1}}\frac{1}{(1-\omega^*)^{d/2}}\sum_{|\k|=n+1}\left(\frac{\omega^*\|\x-\y\|_\infty^2}{1-\omega^*}\right)^{\frac{\vert\k\vert}{2}}\\
&\leq \frac{C_1\|f\|_{\epsilon_{\omegab,\1}}}{(1-\omega^*)^{d/2}}\left(\frac{\omega^*\|\x-\y\|_\infty^2}{1-\omega^*}\right)^\frac{n+1}{2}\frac{(d+n)!}{(d-1)!(n+1)!}.
\end{align*}
Since $\omega^*\|\x-\y\|_\infty^2/(1-\omega^*)<1$ and $(d+n)!/((d-1)!(n+1)!)=\cO(n^{d-1})$, we get that $R_n\rightarrow 0$ as $n$ goes to $\infty$. Thus, $f$ is indeed analytic.
\end{proof}

\subsection{Multivariate integration and tractability in weighted Hermite spaces}\label{ssec:integration}

A survey of the analysis of multivariate integration in reproducing kernel Hilbert spaces is given, e.g., in Dick and Pillichshammer \cite{dick}, where they study the integration of functions over the unit cube $[0,1)^d$. In the following, we intend to study multivariate integration of functions which belong to a Hermite space $\cH_\r$. For that, we consider the integral of functions $f\in\cH_\r$ over the $\R^d$ with respect to the Gaussian measure, i.e.,
\begin{align}\label{eq:integral}
I(f)=\int_{\R^d}f(\x)\varphi_d(\x)d\x.
\end{align}
The linear functional $I$ is well-defined, because a Gaussian square-integrable function $f$ also is integrable with respect to the Gaussian measure. Now we intend to approximate the $d$-dimensional integral \eqref{eq:integral} by a quasi-Monte Carlo rule, i.e., an equally weighted quadrature rule
\begin{align*}
Q_{n,d}(f,\cP)=\frac{1}{n}\sum_{i=1}^{n}f(\x_i)
\end{align*}
with deterministic point set $\cP=\{\x_1,\ldots,\x_n\}\subset\R^d$. For that, it is interesting to analyze the integration error of a QMC rule which depends on both the function $f$ and the point set $\cP$.

\begin{definition}
Let $\cH_\r$ be a Hermite space and let $\cP$ be the quadrature points used in the QMC rule $Q_{n,d}$.
\begin{enumerate}[(i)]
\item For $f\in\cH_\r$ the {\em QMC integration error} of $f$ is given by
\begin{align*}
e(f,\cP):=\vert I(f)-Q_{n,d}(f,\cP)\vert.
\end{align*}
\item The {\em worst case error} for QMC integration in $\cH_\r$ is defined as
\begin{align*}
e_{n,d}(\cP,\cH_\r):=\sup_{f\in\cH_\r,\|f\|_\r\leq1}e(f,\cP)=\sup_{f\in\cH_\r,\|f\|_\r\leq1}\vert I(f)-Q_{n,d}(f,\cP)\vert.
\end{align*}
\end{enumerate}
\end{definition}

Since we know that a Hermite space has a reproducing kernel function $K_\r$, we can use the reproducing property of $K_\r$ to get 
\begin{align*}
\int_{\R^d}f(\x)\varphi_d(\x)d\x=\int_{\R^d}\langle f,K_\r(\cdot,\x)\rangle_\r \varphi_d(\x)d\x=\left\langle f,\int_{\R^d}K_\r(\cdot,\x)\varphi_d(\x)d\x\right\rangle_{\!\r}\,,
\end{align*}
where we have used continuity of integration on the space $\cH_\r$, which
in turn follows from 
\begin{align*}
\left|\int_{\R^d} f(\x) \varphi_d(\x)d\x\right|
=|\hat f(0)|\le \|f\|_\r\,.
\end{align*}
Therefore the {\em representer of integration} is the function $\x\mapsto\int_{\R^d}K_\r(\x,\y)\varphi_d(\y)d\y$. In the same manner we obtain
\begin{align*}
\frac{1}{n}\sum_{i=1}^{n}f(\x_i)=\frac{1}{n}\sum_{i=1}^{n}\langle f, K_\r(\cdot,\x_i)\rangle_{\r}=\left\langle f, \frac{1}{n}\sum_{i=1}^{n}K_\r(\cdot,\x_i)\right\rangle_{\!\r},
\end{align*}
that means, $\x\mapsto\frac{1}{n}\sum_{i=1}^{n}K_\r(\x,\x_i)$ is the representer of the QMC rule. Now, the integration error $e(f,\cP)$ can be estimated by
\begin{align*}
\big|I(f)-Q_{n,d}(f,\cP)\big|&=\left\vert\left\langle f,\int_{\R^d}K_\r(\cdot,\x)\varphi_d(\x)d\x-\frac{1}{n}\sum_{i=1}^{n}K_\r(\cdot,\x_i)\right\rangle\right\vert\\
&\leq \left\|f\right\|_\r\bigg\|\int_{\R^d}K_\r(\cdot,\x)\varphi_d(\x)d\x-\frac{1}{n}\sum_{i=1}^{n}K_\r(\cdot,\x_i)\bigg\|_\r.
\end{align*}
So for any point set $\cP=\{\x_1,\ldots,\x_n\}$ the worst case error $e_{n,d}(\cP,\cH_\r)$ for QMC integration in the Hermite space $\cH_\r$ is given by
\begin{align*}
e_{n,d}(\cP,\cH_\r)=\bigg\|\int_{\R^d}K_\r(\cdot,\x)\varphi_d(\x)d\x-\frac{1}{n}\sum_{i=1}^{n}K_\r(\cdot,\x_i)\bigg\|_\r
\end{align*}
Because of the linearity of the inner product we obtain
\begin{align*}
e^2_{n,d}(\cP,\cH_\r)&=\Big\langle \int_{\R^d}K_\r(\cdot,\x)\varphi_d(\x)d\x-\frac{1}{n}\sum_{i=1}^{n}K_\r(\cdot,\x_i),\int_{\R^d}K_\r(\cdot,\x)\varphi_d(\x)d\x-\frac{1}{n}\sum_{i=1}^{n}K_\r(\cdot,\x_i) \Big\rangle_{\!\r}\\
&=\Big\langle \int_{\R^d}K_\r(\cdot,\x)\varphi_d(\x)d\x,\int_{\R^d}K_\r(\cdot,\x)\varphi_d(\x)d\x\Big\rangle_{\!\r}-\frac{2}{n}\Big\langle \int_{\R^d}K_\r(\cdot,\x)\varphi_d(\x)d\x,\sum_{i=1}^{n}K_\r(\cdot,\x_i) \Big\rangle_{\!\r}\\
&\qquad+\frac{1}{n^2}\Big\langle \sum_{i=1}^{n}K_\r(\cdot,\x_i),\sum_{i=1}^{n}K_\r(\cdot,\x_i)\Big\rangle_{\!\r}
\end{align*}
and further
\begin{align*}
e^2_{n,d}(\cP,\cH_\r)&=\int_{\R^d}\int_{\R^d}\Big\langle K_\r(\cdot,\x),K_\r(\cdot,\y)\Big\rangle_{\!\r}\varphi_d(\x)\varphi_d(\y)d\x d\y-\frac{2}{n}\sum_{i=1}^{n}\int_{\R^d}\Big\langle K_\r(\cdot,\x),K_\r(\cdot,\x_i) \Big\rangle_{\!\r}\varphi_d(\x)d\x\\
&\qquad+\frac{1}{n^2}\sum_{i=1}^{n}\sum_{j=1}^{n}\Big\langle K_\r(\cdot,\x_i),K_\r(\cdot,\x_j)\Big\rangle_{\!\r}.\\
\end{align*}

Using the reproducing property of $K_\r$ we end up with
\begin{align*}
e^2_{n,d}(\cP,\cH_\r)&=\int_{\R^d}\int_{\R^d} K_\r(\x,\y)\varphi_d(\x)\varphi_d(\y)d\x d\y-\frac{2}{n}\sum_{i=1}^{n}\int_{\R^d} K_\r(\x,\x_i)\varphi_d(\x)d\x\\
&\qquad+\frac{1}{n^2}\sum_{i=1}^{n}\sum_{j=1}^{n} K_\r(\x_i,\x_j).
\end{align*}

If we now use that $K_\r(\x,\y)=\sum_{\k\in\N_0^d}\r(\k)H_\k(\x)H_\k(\y)$, we get for the first term
\begin{align*}
\int_{\R^d}\int_{\R^d} K_\r(\x,\y)\varphi_d(\x)\varphi_d(\y)d\x d\y&=\int_{\R^d}\int_{\R^d} \sum_{\k\in\N_0^d}\r(\k)H_\k(\x)H_\k(\y)\varphi_d(\x)\varphi_d(\y)d\x d\y\\
&= \sum_{\k\in\N_0^d}\r(\k)\int_{\R^d}H_\k(\x)\varphi_d(\x)d\x\int_{\R^d}H_\k(\y)\varphi_d(\y)d\y\\
&=\r(\0)
\end{align*}
and for the second term
\begin{align*}
\frac{2}{n}\sum_{i=1}^{n}\int_{\R^d} K_\r(\x,\x_i)\varphi_d(\x)d\x&=\frac{2}{n}\sum_{i=1}^{n}\int_{\R^d} \sum_{\k\in\N_0^d}\r(\k)H_\k(\x)H_\k(\x_i)\varphi_d(\x)d\x\\
&=\frac{2}{n}\sum_{\k\in\N_0^d}\r(\k)\sum_{i=1}^{n}H_\k(\x_i)\int_{\R^d} H_\k(\x)\varphi_d(\x)d\x\\
&=\frac{2}{n}\r(\0)\sum_{i=1}^{n}H_\0(\x_i)=2\r(\0).
\end{align*}
Therefore, we obtain
\begin{align}\label{eq:worstcaseerror}
e_{n,d}(\cP,\cH_\r)&=\left(-\r(\0)+\frac{1}{n^2}\sum_{i=1}^{n}\sum_{j=1}^{n} K_\r(\x_i,\x_j)\right)^{\frac{1}{2}}.
\end{align}

To get an estimate for the integration error, we would like to have an upper bound for the worst case error $e_{n,d}(\cP,\cH_\r)$. Since we have not chosen a specific point set $\cP$, we derive an upper bound through an averaging argument.

\begin{definition}\label{def:gaussianerror}
The Gaussian weighted root-mean-square error for QMC integration is
\begin{align*}
\bar{e}_{n,d}(\cH_\r):=\left(\int_{\R^{dn}}e_{n,d}^{2}(\x_1,\ldots,\x_n;\cH_\r)\varphi_d(\x_1)\ldots\varphi_d(\x_n)d(\x_1,\ldots ,\x_{n})\right)^{\frac{1}{2}}
\end{align*}
where $e_{n,d}(\x_1,\ldots,\x_n;\cH_\r)$ is the worst case error of the point set $\cP=(\x_1,\ldots,\x_n)$. 
\end{definition}

\begin{proposition}\label{prop:gausserror}
Let $\cH_\r$ be a Hermite space. Then the Gaussian weighted root-mean-square error $\bar{e}_{n,d}(\cH_\r)$ for QMC integration is given by
\begin{align*}
\bar{e}_{n,d}(\cH_\r)=\frac{1}{\sqrt{n}}\left(\sum_{\k\in\N_0^d}\r(\k)-\r(\0)\right)^{\frac{1}{2}}.
\end{align*}
\end{proposition}

\begin{proof}
With \eqref{eq:worstcaseerror} the Gaussian weighted root-mean-square error can be written as
\begin{align*}
\bar{e}_{n,d}^2(\cH_\r)=-\r(\0)+\frac{1}{n^2}\sum_{i=1}^{n}\sum_{j=1}^{n}\int_{\R^d}\int_{\R^d}K_\r(\x_i,\x_j)\varphi_d(\x_i)\varphi_d(\x_j)d\x_id\x_j
\end{align*}
Now we consider
\begin{align*}
&\sum_{i=1}^{n}\sum_{j=1}^{n}\int_{\R^d}\int_{\R^d}K_\r(\x_i,\x_j)\varphi_d(\x_i)\varphi_d(\x_j)d\x_id\x_j=\\
&\qquad\qquad=\sum_{i=1}^{n}\int_{\R^d}K_\r(\x_i,\x_i)\varphi_d(\x_i)d\x_i+\sum_{i\neq j}\int_{\R^d}\int_{\R^d}K_\r(\x_i,\x_j)\varphi_d(\x_i)\varphi_d(\x_j)d\x_id\x_j
\end{align*}
in more detail. Since
\begin{align*}
\int_{\R^d}K_\r(\x_i,\x_i)\varphi_d(\x_i)d\x_i&=\sum_{\k\in\N_0^d}\r(\k)\int_{\R^d}H_\k(\x)^2\varphi_d(\x)d\x\\
&=\sum_{\k\in\N_0^d}\r(\k)
\end{align*}
and for $i\neq j$
\begin{align*}
\int_{\R^d}\int_{\R^d}K_\r(\x_i,\x_j)\varphi_d(\x_i)\varphi_d(\x_j)d\x_id\x_j&=\sum_{\k\in\N_0^d}\r(\k)\int_{\R^d}H_\k(\x)\varphi_d(\x)d\x\int_{\R^d}H_\k(\y)\varphi_d(\y)d\y\\
&=\r(\0),
\end{align*}
we have that
\begin{align*}
\sum_{i=1}^{n}\sum_{j=1}^{n}\int_{\R^d}\int_{\R^d}K_\r(\x_i,\x_j)\varphi_d(\x_i)\varphi_d(\x_j)d\x_id\x_j=n \sum_{\k\in\N_0^d}\r(\k)+(n^2-n)\r(\0).
\end{align*}
So we end up with
\begin{align*}
\bar{e}_{n,d}^2(\cH_\r)&=-\r(\0)+\frac{1}{n}\sum_{\k\in\N_0^d}\r(\k)+\frac{n^2-n}{n^2}\r(\0)\\
&=\frac{1}{n}\left(\sum_{\k\in\N_0^d}\r(\k)-\r(\0)\right).
\end{align*}
\end{proof}

\begin{corollary}\label{coro:existence}
Let $\cH_\r$ be Hermite space of $d$-dimensional functions. For each $n\in\N$ there exists a point set $\cP=\{\x_1,\ldots,\x_n\}$ such that for the worst case error for QMC integration in $\cH_{\r}$ the upper bound
\begin{align*}
e_{d,n}(\cP,\cH_\r)\leq \frac{c(d)}{\sqrt{n}}
\end{align*}
holds, where $c(d):=\sqrt{\sum_{\k\in\N_0^d\backslash\{\0\}}\r(\k)}$.
\end{corollary}
\begin{proof}
We can estimate the minimum of the worst case error over all possible point sets $\cP$ by the Gaussian weighted root-mean-square error, i.e.,
\begin{align*}
\min_{\cP=\{\x_1,\ldots,\x_n\}\subset\R^d}e_{n,d}(\cP,\cH_\r) \leq \bar{e}_{n,d}(\cH_\r).
\end{align*}
With Proposition \ref{prop:gausserror} we get
\begin{align*}
\min_{\cP=\{\x_1,\ldots,\x_n\}\subset\R^d}e_{n,d}(\cP) &\leq\frac{c(d)}{\sqrt{n}}.
\end{align*}
\end{proof}

Corollary \ref{coro:existence} states that a sequence of point sets $(\cP_n)_{n\in\N}$ can be chosen such that the convergence rate of the corresponding quasi-Monte Carlo method is $\frac{1}{2}$, which is the convergence rate of crude Monte Carlo simulation. However, this upper bound results from an averaging argument and, consequently, we are confident that the upper bound can be improved for specific point sets. Furthermore, $c(d)$ is constant with respect to $n$ but it still can depend on the dimension $d$ and it depends on the function $\r$. To get more information about the behavior of $c(d)$, we need additional information about the function $\r$. For that, we will consider the two examples of Hermite spaces, given in the previous subsection.

To study the dependence on dimension $d$ it is reasonable to consider the minimal number of information operations which are required to reduce the initial error, given by
\begin{align*}
e_{0,d}(\cH_\r)=\sup_{f\in\cH_\r,\|f\|_\r\leq1}\vert I(f)\vert,
\end{align*}
by a factor $\varepsilon\in(0,1)$. Since $I(f)=\hat{f}(\0)=\langle f,1\rangle_\r$, we know that the initial error is $\r(\0)$. Now let
\begin{align*}
n_{\min}(\varepsilon,d;\cH_\r):=\min\{n\,:\,\exists\,\cP \textnormal{ such that } e_{n,d}(\cP,\cH_\r)\leq\varepsilon\r(\0)\},
\end{align*}
which is called the {\em information complexity} of quasi-Monte Carlo integration in the Hermite space $\cH_\r$.

\begin{definition}
\begin{enumerate}
	\item Multivariate integration in the Hermite space $\cH_\r$ is called \emph{polynomially tractable} if there exist $c,p,q\in\R^{+}$ such that
\begin{align}\label{def:tractability}
n_{\min}(\varepsilon,d;\cH_\r)\leq c\, d^{q}\, \varepsilon^{-p}
\end{align}
holds for all dimensions $d\in\N$ and for all $\varepsilon\in(0,1)$. The infima of the numbers $p$ and $q$ are called $\varepsilon$- and $d$-exponents of polynomial tractability.
\item If \eqref{def:tractability} holds with $q=0$, we say that multivariate integration in $\cH_\r$ is strongly polynomially tractable. The infimum of $p$ is called $\varepsilon$-exponent of strong polynomial tractability
\item If \eqref{def:tractability} does not hold, we say that multivariate integration in $\cH_\r$ is polynomially intractable.
\end{enumerate}
\end{definition}

We start with the Hermite space $\cH_{\p_{\alphab,\gammab}}$ of functions with polynomially decaying Hermite coefficients and we will show that the choice of the weights $\gammab$ is crucial for obtaining polynomial tractability. 

\begin{proposition}
Let $\gammab=(\gamma_1,\ldots,\gamma_d)\in(0,\infty)^d$ be a non-increasing sequence of weights and $\alphab=(\alpha_1,\ldots,\alpha_d)\in (1,\infty)^d$ a sequence of smoothness parameters. There exists a quasi-Monte Carlo method with point set $\cP$ such that the worst case error for QMC integration in the Hermite space $\cH_{\p_{\alphab,\gammab}}$ of functions with polynomially decaying coefficients  can be bounded from above by
\begin{align}\label{eq:polyupperbound}
e_{n,d}(\cP,\cH_{\p_{\alphab,\gammab}})\leq \frac{1}{\sqrt{n}}\,e^{\frac{\zeta(\alpha_{\min})}{2}\sum_{j=1}^{d}\gamma_j}
\end{align}
with $\alpha_{\min}:=\min\{\alpha_i: i=1,\ldots,d\}$.
\end{proposition}
\begin{proof}
From Corollary \ref{coro:existence} we know that there exists a quasi-Monte Carlo method such that the worst case error can be estimated by
\begin{align*}
e_{n,d}(\cP,\cH_{\p_{\alphab,\gammab}})\leq\frac{1}{\sqrt{n}}\left(\sum_{\k\in\N_0^d}\p_{\alphab,\gammab}(\k)-1\right)^{\frac{1}{2}}.
\end{align*}
Due to \eqref{eq:polysum} we get
\begin{align*}
e_{n,d}(\cP,\cH_{\p_{\alphab,\gammab}})&\leq \frac{1}{\sqrt{n}}\left(\prod_{j=1}^{d}\left(1+\gamma_j\zeta(\alpha_j)\right)\right)^{\frac{1}{2}}\\
&=\frac{1}{\sqrt{n}}\,e^{\frac{1}{2}\sum_{j=1}^{d}\ln\left(1+\gamma_j\zeta(\alpha_j)\right)}\\
&\leq\frac{1}{\sqrt{n}}\,e^{\frac{1}{2}\sum_{j=1}^{d}\gamma_j\zeta(\alpha_j)}.
\end{align*}
Denote the smallest smoothness parameter as $\alpha_{\min}$. Then the statement follows,
\begin{align*}
e_{n,d}(\cP,\cH_{\p_{\alphab,\gammab}})&\leq\frac{1}{\sqrt{n}}\,e^{\frac{\zeta(\alpha_{\min})}{2}\sum_{j=1}^{d}\gamma_j}.
\end{align*}
\end{proof}

The upper bound \eqref{eq:polyupperbound} allows us to give sufficient conditions on the weights concerning the tractability of multivariate integration in the Hermite space $\cH_{\p_{\alphab,\gammab}}$. Let $\gammab$ be an infinite non-increasing sequence of positive weights and $\alphab$ be an infinite sequence of smoothness parameters. Furthermore, assume that there exist a $\alpha_{\min}>1$ such that $\alpha_j>\alpha_{\min}$ holds for all $j\in\N$. From \eqref{eq:polyupperbound} we obtain that
\begin{align*}
n_{\min}(\varepsilon,d,\cH_{\p_{\alphab,\gammab}})\leq \varepsilon^{-2} e^{\zeta(\alpha_{\min})\sum_{j=1}^{d}\gamma_j}.
\end{align*}
In the case that $\sum_{j=1}^{\infty}\gamma_j<\infty$ holds, the upper bound of the information complexity does not depend on the dimension $d$ and thus we get a sufficient condition for strong polynomial tractability. If we suppose that $\limsup_{d}\sum_{j=1}^{d}\gamma_j /\ln(d)<\infty$, 
\begin{align*}
n_{\min}(\varepsilon,d;\cH_{\p_{\alphab,\gammab}})\leq \varepsilon^{-2} d^{\,\zeta(\alpha_{\min})\sum_{j=1}^{d}\gamma_j/\ln(d)}.
\end{align*}

\begin{theorem}\label{th:polytract}
Let $\gammab=(\gamma_1,\gamma_2,\ldots)\in(0,\infty)^\N$ be a non-increasing sequence of weights and $\alphab=(\alpha_1,\alpha_2,\ldots)\in (1,\infty)^\N$ be a sequence of smoothness parameters such that $\inf_d \alpha_d>1$.  Multivariate integration in the weighted Hermite space $\cH_{\p_{\alphab,\gammab}}$ of functions with polynomially decaying coefficients is
\begin{enumerate}
	\item strongly polynomially tractable, if $\sum_{j=1}^{\infty}\gamma_j<\infty$.
	\item polynomially tractable, if $A:=\limsup_d\frac{\sum_{j=1}^{d}\gamma_j}{\ln(d)}<\infty$.
\end{enumerate}
\end{theorem}

In both cases the $\varepsilon$-exponent is at most $2$, which comes from the averaging argument where we used the Gaussian-weighted root-mean-square error. However, we are convinced that there are quasi-Monte Carlo methods for which the $\varepsilon$-exponent can be improved. The $d$-exponent is given by $\zeta(\alpha_{\min})A$ with $A$ defined in Theorem \ref{th:polytract} and so it depends on both the weights and the smoothness parameters.

For the Hermite space with exponentially decaying coefficients we can show not only an upper bound but also a lower bound. With it, we are able to state sufficient and necessary conditions on the weight sequence for polynomial tractability.

\begin{proposition}
Let $\gammab=(\gamma_1,\ldots,\gamma_d)\in(0,\infty)^d$ be a non-increasing sequence of weights and $\omegab=(\omega_1,\ldots,\omega_d)\in (0,1)^d$ a sequence of smoothness parameters. There exists a quasi-Monte Carlo method with point set $\cP$ such that the worst case error for QMC integration in the Hermite space $\cH_{\epsilonb_{\omegab,\gammab}}$ of functions with exponentially decaying coefficients can be bounded from above by
\begin{align}\label{eq:expupperbound}
e_{n,d}(\cP,\cH_{\epsilonb_{\omegab,\gammab}})\leq \frac{1}{\sqrt{n}}\,e^{\frac{\omega_{\max}}{2(1-\omega_{\max})}\sum_{j=1}^{d}\gamma_j}
\end{align}
with $\omega_{\max}:=\max\{\omega_i: i=1,\ldots,d\}$.
\end{proposition}
\begin{proof}
With Corollary \ref{coro:existence} and the definition of $\epsilonb_{\omegab,\gammab}$, see \eqref{def:exp}, we know that there exists a QMC method with
\begin{align*}
e_{n,d}(\cP,\cH_{\epsilonb_{\omegab,\gammab}})&\leq \frac{1}{\sqrt{n}}\,e^{\frac{1}{2}\sum_{j=1}^{\infty}\gamma_j\frac{\omega_j}{1-\omega_j}}.
\end{align*}
The statement directly follows, because $\omega_j/(1-\omega_j)<\omega_{\max}/(1-\omega_{\max})$ for all $j=1,\ldots,d$.
\end{proof}

We again get sufficient conditions to the weight sequence to achieve polynomial tractability in $\cH_{\epsilonb_{\omegab,\gammab}}$. We have that multivariate integration is strong polynomial tractable if the weight sequence is summable and the integration problem is polynomial tractable if $\limsup_d \sum_{j=1}\gamma_j/\ln(d)<\infty$. However, to show that these conditions are not only sufficient but also necessary, we need a lower bound on the worst case error. For that, we have to choose our weights small enough such that the reproducing kernel $K_{\epsilonb_{\omegab,\gammab}}$ is positive.

\begin{proposition}\label{prop:lowerbound}
Let $\gammab=(\gamma_1,\ldots,\gamma_d)\in(0,1)^d$ be a non-increasing sequence of weights and $\omegab=(\omega_1,\ldots,\omega_d)\in (0,1)^d$ a sequence of smoothness parameters. For any quasi-Monte Carlo method with point set $\cP$ the worst case error for QMC integration in the Hermite space $\cH_{\epsilon_{\omegab,\gammab}}$ of functions with exponentially decaying coefficients can be bounded from below by
\begin{align}\label{eq:explowerbound}
e_{n,d}(\cP,\cH_{\epsilonb_{\omegab,\gammab}})\geq \left(-1 +\frac{\prod_{j=1}^{d}\left(1+\gamma_jc(\omega_j)\right)}{n}\right)^{\frac{1}{2}}
\end{align}
with $c(\omega)=\frac{1-\sqrt{1-\omega^2}}{\sqrt{1-\omega^2}}$.
\end{proposition}
\begin{proof}
With \eqref{eq:worstcaseerror} we have that
\begin{align*}
e_{n,d}^2(\cP,\cH_{\epsilonb_{\omegab,\gammab}})=-1+\frac{1}{n^2}\sum_{i,j=1}^{n}K_{\epsilonb_{\omegab,\gammab}}(\x_i,\x_j)
\end{align*}
holds for any QMC method with $\cP=\{\x_1,\ldots,\x_n\}$. Since all weights $\gamma_j$ are smaller than $1$, we know from \eqref{eq:kernelexp} that $K_{\epsilonb_{\omegab,\gammab}}(\x_i,\x_j)>0$ and so we neglect all terms with $i\neq j$. Hence,
\begin{align*}
e_{n,d}^2(\cP,\cH_{\epsilonb_{\omegab,\gammab}})\geq -1+\frac{1}{n}\sum_{i=1}^{n}K_{\epsilonb_{\omegab,\gammab}}(\x_i,\x_i).
\end{align*}
For any $i=1,\ldots,d$ we get from \eqref{eq:kernelexp} that
\begin{align*}
K_{\epsilonb_{\omegab,\gammab}}(\x_i,\x_i)&=\prod_{j=1}^{d}\left(1-\gamma_j+\gamma_j\frac{1}{\sqrt{1-\omega_j^2}}\exp\left(\frac{\omega_j}{1+\omega_j}x_{i,j}^2\right)\right)\\
&\geq \prod_{j=1}^{d}\left(1-\gamma_j+\gamma_j\frac{1}{\sqrt{1-\omega_j^2}}\right)
\end{align*}
where $x_{i,j}$ is the $j$-th entry of $\x_i$. Putting everything together, we get 
\begin{align*}
e_{n,d}^2(\cP,\cH_{\epsilonb_{\omegab,\gammab}})\geq-1+\frac{1}{n}\prod_{j=1}^{d}\left(1+\gamma_j\frac{1-\sqrt{1-\omega_j^2}}{\sqrt{1-\omega_j^2}}\right).
\end{align*}
\end{proof}

Now we assume that $\gammab$ is a non-increasing infinite sequence of weights $\gamma_j\in(0,1)$. Moreover, we suppose that $\omegab\in(0,1)^\N$ is a sequence such that $\omega_j\ge\omega_{\min}>0$ for all $j\in\N$. It follows from Proposition \ref{prop:lowerbound} that
\begin{align*}
n_{\min}(\varepsilon,d;\cH_{\epsilonb_{\omegab,\gammab}})&\geq\frac{\prod_{j=1}^{d}\left(1+\gamma_jc(\omega_j)\right)}{\varepsilon^2+1}\\
&\geq\frac{\prod_{j=1}^{d}\left(1+\gamma_jc(\omega_{\min})\right)}{\varepsilon^2+1}
\end{align*}
with $c(\omega)=\frac{1-\sqrt{1-\omega^2}}{\sqrt{1-\omega^2}}$. Under the assumption that $\gamma_j\ge\gamma_{\min}>0$ for all $j\in\N$, we get
\begin{align*}
n_{\min}(\varepsilon,d;\cH_{\epsilonb_{\omegab,\gammab}})&\geq\frac{\left(1+\gamma_{\min}c(\omega_{\min})\right)^d}{\varepsilon^2+1}
\end{align*}
and thus the information complexity grows exponentially in $d$. If we assume that $\gamma_j\rightarrow0$ and $\sum_{j=1}^{\infty}\gamma_j=\infty$, we can deduce from
\begin{align*}
\prod_{j=1}^{d}\left(1+\gamma_jc(\omega_j)\right)\geq e^{\tilde{c}c(\omega_0)\sum_{j=1}^{d}\gamma_j}
\end{align*}
with constant $\tilde{c}>0$, that $n_{\min}(\varepsilon,d;\cH_{\epsilonb_{\omegab,\gammab}})\rightarrow \infty$ as $d\rightarrow \infty$. Therefore $\sum_{j=1}^{\infty}\gamma_j<\infty$ is a necessary condition for strong polynomial tractability. 
Next we assume that $\gamma_j\rightarrow0$ for all $j\in\N$ and $\limsup_d \sum_{j=1}^{d}\gamma_j/\ln(d)=\infty$. Then we can use the estimate
\begin{align*}
\prod_{j=1}^{d}\left(1+\gamma_jc(\omega_j)\right)\geq d^{\,\bar{c}c(\omega_0)\sum_{j=1}^{d}\gamma_j/\ln(d)}
\end{align*}
with constant $\bar{c}>0$ to conclude that $n_{\min}$ grows faster than any
power of $d$. Thus we can conclude that $\limsup_d
\sum_{j=1}^{d}\gamma_j/\ln(d)<\infty$ is necessary to get polynomial
tractability. Now we summarize the results concerning the Hermite space with
exponentially decaying coefficients in the following theorem.

\begin{theorem}
Let $\gammab=(\gamma_1,\gamma_2,\ldots)\in(0,1)^\N$ be a non-increasing sequence of weights and $\omegab=(\omega_1,\omega_2,\ldots)\in (0,1)^\N$ be a sequence of smoothness parameters such that $\liminf_d \omega_d>0$ and $\limsup_d \omega_d<1$. Multivariate integration in the weighted Hermite space $\cH_{\epsilon_{\omegab,\gammab}}$ of functions with exponentially decaying coefficients is 
\begin{enumerate}
	\item strong polynomially tractable iff $\sum_{j=1}^{\infty}\gamma_j<\infty$,
	\item polynomially tractable iff $\limsup_d\frac{\sum_{j=1}^{d}\gamma_j}{\ln(d)}<\infty$.
\end{enumerate}
\end{theorem}

\section{Quasi-Monte Carlo integration and orthogonal transforms}\label{sec:orthogonal}

In this section we study the influence of orthogonal transforms on the efficiency of QMC algorithm. 

\subsection{Error analysis}\label{ssec:error}

First, let us summarize the results which we have until now. Due to
\eqref{eq:errorbound1} we know that the error of the quasi-Monte Carlo
algorithm is bounded by the sum of the discretization error and the integration
error. For the discretization error we have the estimate
\eqref{eq:discerrorbound} which states that the error can be reduced by
increasing the fineness of the discretization grid. The convergence rate of the
discretization step depends on the given time-continuous problem as well as on
the discretization method. Moreover, we denote by $f_d$ the $d$-dimensional
function which we obtain after the discretization and we assume that for every
$d$ the function $f_d$ belongs to a weighted Hermite space in which
multivariate integration is polynomially tractable. Then,
\begin{align}\label{eq:error}
\vert\E(g(B))-Q_{n,d}\vert\leq c_1d^{-q_1}+\|f_d\|_\r\, c_2d^{q_2}n^{-\frac{1}{2}}
\end{align}
with constants $c_1,c_2>0$. The exponent $q_1>0$ corresponds to the discretization method and $q_2>0$ is the $d$-exponent. However, it can happen that the norm of $f_d$ grows in $d$, i.e., the sequence $(\|f_d\|_\r)_{d\in\N}\rightarrow\infty$ as $d\rightarrow\infty$. Then, even if we have strong polynomial tractability in the Hermite space, i.e., $q_2=0$, the second term still gets bigger with increasing $d$, sometimes rendering bound \eqref{eq:error} useless.

So let us revisit the estimation of of the integration error. Because of the change-of-variable formula for multi-dimensional integration we have for any orthogonal transform $U:\R^d\longrightarrow\R^d$ and for any $f_d\in L^2(\R^d,\varphi_d)$ the identity
\begin{align*}
\int_{\R^d}f_d(\x)\varphi_d(\x)d\x=\int_{\R^d}f_d(U\x)\varphi_d(\x)d\x.
\end{align*}
Thus, instead of approximating the integral on the left-hand side we can
compute the integral on the right-hand side by a quasi-Monte Carlo rule. Then
we can replace the norm of $f_d$ by the norm of $f_d\circ U$ in the upper bound
of the integration error in \eqref{eq:error}. Now the idea is to choose the
orthogonal transform $U$ such that $\|f_d\circ U\|_\r$ increases slower in $d$
than the norm of $f_d$ does or to even get that $\|f_d\circ U\|_\r$ is bounded.

For that, we would like to know how an orthogonal transform affects the norm $\|\cdot\|_{\r}$ of the Hermite space $\cH_\r$, but at first we consider an example to show what one can achieve by applying an orthogonal transform to the integration problem.

\begin{example}
We are interested in the computation of $\E(\exp(B_1))$ where $B=(B_t)_{t\in[0,1]}$ is a standard Brownian motion. Note that this problem is of the form of the model problem presented in the introduction. The first attempt is to apply the forward method to discretize the Brownian motion on the time grid $\frac{1}{d},\frac{2}{d},\ldots,\frac{d}{d}$. Then the discretized problem is given by the $d$-dimensional function $f_d(\x)=\exp(\frac{1}{\sqrt{d}}\sum_{j=1}^{d}x_j)$ which belongs to $\cH_{\p_{\alphab,\gammab}}$. We further assume that $\alphab$ is a $d$-dimensional vector of integers larger than $1$. For the $\k$-th Hermite coefficient we get, by using the exponential generating function,
\begin{align*}
\widehat{f_d}(\k)&=\int_{\R^d}e^{\frac{1}{\sqrt{d}}\sum_{j=1}^{d}x_j}H_\k(\x)\varphi_d(\x)d\x\\
&=\prod_{j=1}^{d}\int_{\R}e^{\frac{1}{\sqrt{d}}x_j}H_{k_j}(x_j)\varphi(x_j)dx_j\\
&=e^{\frac{1}{2}}\prod_{j=1}^{d}\int_{\R}\sum_{\ell=0}^{\infty}\frac{1}{\sqrt{\ell!d^{\ell}}}H_\ell(x_j)H_{k_j}(x_j)\varphi(x_j)dx_j\\
&=e^{\frac{1}{2}}\prod_{j=1}^{d}\frac{1}{\sqrt{k_j!\,d^{k_j}}}\\
&=e^{\frac{1}{2}}\frac{1}{\sqrt{\k!\,d^{|\k|}}}.
\end{align*}
The norm of $f_d$ is given by
\begin{align*}
\|f_d\|_{\p_{\alphab,\gammab}}^2&=e\prod_{j=1}^{d}\left(1+\gamma_j^{-1}\sum_{k=1}^{\infty}k^{\alpha_j}\frac{1}{k!d^k}\right)\\
&=e\prod_{j=1}^{d}\left(1+\gamma_j^{-1}\left(\frac{1}{d}m_{\alpha_j}(1/d)e^{1/d}\right)\right)\\
&=e^{1+\sum_{j=1}^{d}\ln\left(1+\gamma_j^{-1}\left(\frac{1}{d}m_{\alpha_j}(1/d)e^{1/d}\right)\right)}
\end{align*}
where $m_\alphab$ is a polynomial of degree $\alpha-1$. Moreover, $m_\alphab(1/d)$ is monotonically decreasing in $d$ and $m_\alpha(0)=1$. If we now choose the weights $\gammab$ as $\gamma_j=j^{-2}$, we would achieve that integration in $\cH_{\p_{\alphab,\gammab}}$ is strong polynomially tractable, due to Theorem \ref{th:polytract}. However, we can bound $\|f_d\|_{\p_{\alphab,\gammab}}$ from below,
\begin{align*}
\|f_d\|_{\p_{\alphab,\gammab}}^2&\geq e^{1+\sum_{j=1}^{d}\ln\left(\frac{j^2}{d}\right)}\\
&=e^{1+(2\ln(d!)-d\ln(d))}\\
&=e \frac{(d!)^2}{d^d}.
\end{align*} 
Since $\frac{(d!)^2}{d^d}\rightarrow\infty$ as $d\rightarrow\infty$, the
sequence of norms $(\|f_d\|_{\p_{\alphab,\gammab}})_{d\in\N}$ grows in $d$.
Thus, even though we have strong polynomial tractability of the integration
problem the upper bound of the integration error is not independent of
the dimension, because the norm still depends on $d$. 

Now let us revisit the problem: Instead of the forward method we use the
Brownian bridge construction method to discretize the Brownian motion on the
time grid $\frac{1}{d},\frac{2}{d},\ldots,\frac{d}{d}$. In other words, we
apply an orthogonal transform $U$, namely the inverse Haar transform, to the
integration problem. Then the discretized problem is given by the function
$(f_d\circ U)(\x)=\exp(x_1)$ which belongs to the same Hermite space as
$f_d$. The $\k$-the Hermite coefficient is given by \begin{align*}
\widehat{f_d\circ U}(\k)=\begin{cases}\sqrt{\frac{e}{k_1!}}&\textnormal{if } k_2=\ldots=k_d=0\\0&\textnormal{else}\end{cases}
\end{align*}
and consequently,
\begin{align*}
\|f_d\circ U\|_{\p_{\alphab,\gammab}}^2&=e+\gamma_1^{-1}\sum_{k=1}^{\infty}k_1^{\alpha_1}\frac{e}{k_1!}\\
&=e\left(1+\gamma_1^{-1}m_{\alpha_1}(1)e\right).
\end{align*}
This means that the upper bound of the integration error does not depend on $d$ if we choose our weights $\gammab$ such that strong polynomial tractability holds.
\end{example}

\subsection{Orthogonal transforms}\label{ssec:orthogonal}

As we have seen above, the convergence of a quasi-Monte Carlo algorithm can be influenced by applying an orthogonal transform to the integration problem. However, it is crucial to choose the orthogonal transform tailored to the integration problem to improve the efficiency of the quasi-Monte Carlo algorithm, but in general this is not an easy problem at all. A first attempt to get a better understanding of the behavior of the orthogonal transforms is to study the relation of the Hermite coefficients of $f$ and $f\circ U$.

For an orthogonal transform $U$ of the $\R^d$ we define the mapping $\cA_U:L^2(\R^d,\varphi_d)\longrightarrow L^2(\R^d,\varphi_d)$ as the composition of $f$ and $U$, i.e., $\cA_U f:=f\circ U$. 

\begin{theorem}\label{th:automorphism}
Let $U$ be an orthogonal transform of the $\R^d$. The mapping $\cA_U$ is a Hilbert space automorphism on $L^2(\R^d,\varphi_d)$ and furthermore $(\cA_U)^{-1}=\cA_{U^{-1}}$.
\end{theorem}
\begin{proof}
Let $f,g\in L^2(\R^d,\varphi_d)$. Then
\begin{align*}
\langle \cA_Uf,\cA_Ug\rangle_{L^2(\R^d,\varphi_d)}&=\int_{\R^d}(\cA_Uf)(\x)(\cA_Ug)(\x)\varphi_d(\x)d\x\\
&=\int_{\R^d}f(U\x)g(U\x)\varphi_d(\x)d\x\\
&=\int_{\R^d}f(\x)g(\x)\varphi_d(\x)d\x\\
&=\langle f,g\rangle_{L^2(\R^d,\varphi_d)}
\end{align*}
and so we have that $\cA_U$ preserves the inner product. Moreover, we have that $\cA_U$ is linear as well as surjective. Thus, $\cA_U$ is a Hilbert space automorphism on $L^2(\R^d,\varphi_d)$. Since for any $\x\in\R^d$
\begin{align*}
\cA_U\cA_{U^{-1}}f(\x)=f(UU^{-1}\x)=f(U^{-1}U\x)=\cA_{U^{-1}}\cA_Uf(\x),
\end{align*}
the inverse of $\cA_U$ is given by $A_{U^{-1}}$.
\end{proof}

Apart from this, we look for a nice representation of $\cA_U$. For that, we define sub-spaces of the $L^2(\R^d,\varphi_d)$ as the linear span of the Hermite polynomials of the same degree, i.e., $\cH^{(m)}:=\lin\{H_{\k}(\x): |\k|=m\}$ for any $m\in\N_0$. Since the Hermite polynomials form a basis of the $L^2(\R^d,\varphi_d)$, it is easy to deduce that
\begin{align*}
L^2(\R^d,\varphi_d)=\overline{\bigoplus_{m=0}^{\infty}\cH^{(m)}}.
\end{align*} 
Recall that the $\k$-th Hermite polynomial is given by
\begin{align*}
H_\k(\x)=\frac{\partial^{|\k|}}{\partial \t^\k}G(\x,\t)\big\vert_{\t=\0}
\end{align*}
where $G$ is the exponential generator function of the Hermite polynomials. For any orthogonal transform $U$ of the $\R^d$ we know that 
\begin{align*}
G(U\x,\t)&=e^{(U\x)\tr\t-\frac{\t\tr\t}{2}}\\
&=e^{\x\tr(U\tr\t)-\frac{(U\tr\t)\tr(U\tr\t)}{2}}\\
&=G(\x,U\tr\t)
\end{align*}
holds, and so we obtain
\begin{align*}
\frac{\partial}{\partial t_j}G(U\x,\t)\Big\vert_{\t=\0}&=\frac{\partial}{\partial t_j}G(\x,U\tr\t)\Big\vert_{\t=\0}\\
&=\sum_{k=1}^{d}U_{jk}\left(\frac{\partial}{\partial t_k}G(\x,\t)\Big\vert_{\t=\0}\right).
\end{align*}
In the same way, we have for a sequence $(\beta_1,\ldots,\beta_m)\in\{1,\ldots,d\}^m$ of indices
\begin{align}\label{eq:transformgenerating}
\frac{\partial}{\partial t_{\beta_m}}\cdots\frac{\partial}{\partial t_{\beta_1}}G(U\x,\t)\Big\vert_{\t=\0}=\sum_{\xi_1,\ldots,\xi_m=1}^{d}U_{\beta_m \xi_m}\cdots U_{\beta_1 \xi_1}\left(\frac{\partial}{\partial t_{\xi_m}}\cdots\frac{\partial}{\partial t_{\xi_1}}G(\x,\t)\Big\vert_{\t=\0}\right).
\end{align}
Since there are only derivatives of order $m$ involved, $H_{\k}(U\x)$ with $|\k|=m$ can be written as a linear combination of Hermite polynomials of degree $m$. Therefore, we can state the following lemma. 

\begin{lemma}
The restriction of $\cA_U$ to $\cH^{(m)}$, denoted by $\cA_{U,m}$, is a Hilbert
space automorphism of $\cH^{(m)}$.  
\end{lemma}

Note that the sum in \eqref{eq:transformgenerating} does not represent the left-hand side in a unique way, because there are many $\betab\in\{1,\ldots,d\}^m$ which correspond to the same multi-index. So we consider the function $S:\{1,\ldots,d\}^m\longrightarrow\N_0^d$, given by
\begin{align*}
S((\beta_1,\ldots,\beta_m)):=\left(\#\{k: \beta_k=1\},\ldots,\#\{k: \beta_k=d\}\right),
\end{align*}
which gives for each sequence $(\beta_1,\ldots,\beta_m)$ the corresponding multi-index.

\begin{lemma}\label{lem:sfunction}
Let $\k\in\N_0^d$ be a multi-index with $|\k|=m$. Then 
\begin{align*}
\#\{\betab\in\{1,\ldots,d\}^m: S(\betab)=\k\}=\frac{|\k|!}{\k!}.
\end{align*}
\end{lemma}
\begin{proof}
This can be seen by noting that the right-hand side is just a multinomial
coefficient, $\frac{|\k|!}{\k!}=\binom{k_1+\ldots+k_m}{k_1,\ldots, k_m}$, and by
using elementary combinatorics.
\end{proof}

To use the representation of equation \eqref{eq:transformgenerating} we want a space that takes account of the order of differentiation,
\begin{align*}
\cK^{(m)}:=\ell^2(\{1,\ldots,d\}^m)=\{w:\{1,\ldots,d\}^m\longrightarrow\R\}
\end{align*}
with norm
\begin{align*}
\|w\|_{\ell^2}=\sum_{\betab\in\{1,\ldots,d\}^m}w_\betab^2\,.
\end{align*}
Denote the canonical basis of $\cK^{(m)}$ by $(b_\betab)_{\betab\in\{1,\ldots,d\}^m}$. Next we define the linear operator,
\begin{align*}
\cJ_{m}:\,&\cH^{(m)}\longrightarrow\cK^{(m)}\\
&H_\k\longmapsto \sqrt{\frac{\k!}{|\k|!}}\sum_{S(\betab)=\k}b_\betab\,,
\end{align*}
the adjoint operator of which is given by
\begin{align*}
\cJ_{m}^*:\,&\cK^{(m)}\longrightarrow\cH^{(m)}\\
&b_\betab\longmapsto \sqrt{\frac{S(\betab)!}{|S(\betab)|!}}H_{S(\betab)}.
\end{align*}

\begin{proposition}
The linear operator $\cJ_{m}$ is an isometry and consequently, $\cJ_{m}^{*}\cJ_{m}=\id_{\cH^{m}}$
\end{proposition}
\begin{proof}
Let $f\in\cH^{(m)}$ with $f(\x)=\sum_{|\k|=m}f_{\k}H_\k(\x)$. Then,
\begin{align*}
\|f\|_{\cH^{(m)}}=\sum_{|\k|=m}f_{\k}^2
\end{align*}
and
\begin{align*}
\cJ_{m}f&=\sum_{|\k|=m}f_\k\cJ_m H_\k=\sum_{|\k|=m}f_\k\sqrt{\frac{\k!}{|\k|!}}\sum_{S(\betab)=\k} b_\betab\\
&=\sum_{\betab\in\{1,\ldots,d\}^m}\sum_{\k=S(\betab)}f_\k\sqrt{\frac{\k!}{|\k|!}}\,b_\betab\\
&=\sum_{\betab\in\{1,\ldots,d\}^m}f_{S(\betab)}\sqrt{\frac{S(\betab)!}{|S(\betab)|!}}\,b_\betab.
\end{align*}
So we have
\begin{align*}
\|\cJ_m f\|_{\cK^{(m)}}&=\sum_{\betab\in\{1,\ldots,d\}^m}f_{S(\betab)}^2\frac{S(\betab)!}{|S(\betab)|!}\\
&=\sum_{|\k|=m}\sum_{S(\betab)=\k}f_{S(\betab)}^2\frac{S(\betab)!}{|S(\betab)|!}\\
&=\sum_{|\k|=m}f_{\k}^2\sum_{S(\betab)=\k}\frac{S(\betab)!}{|S(\betab)|!},
\end{align*}
With Lemma \ref{lem:sfunction} we get that $\sum_{S(\betab)=\k}\frac{S(\betab)!}{|S(\betab)|!}=1$ and consequently, $\|f\|_{\cH^{(m)}}=\|\cJ_m f\|_{\cK^{(m)}}$.
\end{proof}

According to equation \eqref{eq:transformgenerating} the application of an orthogonal transform $U$ can be represented in the space $\cK^{(m)}$ as a matrix-vector multiplication with the matrix given by the $m$-fold Kronecker product of $U$: 
\begin{align*}
U^{\otimes m}:=\underbrace{U\otimes\cdots\otimes U}_{m \textnormal{ times}}.
\end{align*}

Therefore, the mapping $\cA_{U,m}$ is of the form
\begin{align*}
\cA_{U,m}=\cJ_{m}^{*} U^{\otimes m}\cJ_{m}
\end{align*}
which can be illustrated by the commutative diagram:

\begin{align*}
\begin{CD}
\cH^{(m)}@>~\cA_{U,m}~>>\cH^{(m)}\\
@V\cJ_{m}VV				@AA\cJ_{m}^{*}A\\
\cK^{(m)}@>>~U^{\otimes m}~>\cK^{(m)}
\end{CD}
\end{align*}

So we have a nice representation of $\cA_{U,m}$ on the space spanned by the Hermite polynomials of degree $m$ and we can deduce a similar representation of $\cA_U$ on $L^2(\R^d,\varphi_d)$. This can be achieved by using
\begin{align*}
\cA_U=\bigoplus_{m=0}^{\infty}\cA_{U,m}.
\end{align*}
We define the space $\cK:=\overline{\bigoplus_{m=0}^{\infty}\cK^{(m)}}$ and the isometry $\cJ:L^2(\R^d,\varphi_d)\longrightarrow\cK$ by
\begin{align*}
\cJ:=\bigoplus_{m=0}^{\infty}\cJ_{m}\,,
\end{align*}
the adjoint operator of which is given by
\begin{align*}
\cJ^*:=\bigoplus_{m=0}^{\infty}\cJ_{m}^*\,.
\end{align*}
This means that we can represent the mapping $A_U$ by
\begin{align}\label{eq:mapping}
\cA_U=\cJ^*\bigoplus_{m=0}^{\infty}U^{\otimes m} \cJ
\end{align}
and we have the commutative diagram:
\begin{align*}
\begin{CD}
L^2(\R^d,\varphi_d)@>~~~~~\cA_{U}~~~~~>>L^2(\R^d,\varphi_d)\\
@V\cJ VV				@AA \cJ^{*}A\\
\cK @>>~\bigoplus_{m\geq0}U^{\otimes m}~>\cK
\end{CD}
\end{align*}
In addition to being an elegant representation of $\cA_U$, equation 
\eqref{eq:mapping} provides us with a formula for computing
the Hermite coefficients of $f\circ U$ from those of $f$. 
For a better illustration we consider an example.

\begin{example}
Let $f:\R^2\longrightarrow\R$ with Hermite expansion $f(\x)=\sum_{\k\in\N_0^2}\widehat{f}(\k)H_\k(\x)$ and $U$ be an orthogonal transform with matrix representation
\begin{align*}
U=\begin{pmatrix}U_{11}&U_{12}\\U_{21}&U_{22}\end{pmatrix}.
\end{align*}
Since we know that $\cA_U$ factors nicely with respect to the spaces $\cH^{(m)}$, we restrict us to the case $m=2$ and consider $\cA_{U,2}$ as well as the projection $f_2$ of $f$ onto the space $\cH^{(2)}$, i.e., $f_2(\x)=f(\x)\vert_{\cH^{(2)}}=\sum_{\vert\k\vert=2}\widehat{f}(\k)H_\k(\x)$. Thus $f_2$ is fully given by the vector of Hermite coefficients $(\widehat{f}(2,0),\widehat{f}(1,1),\widehat{f}(0,2))^\top$. Now we get the Hermite coefficients of $\cA_{U,2}f_2$ by standard matrix-vector calculus. The operator $\cJ_2$, mapping from $\cH^{(2)}$ to $\cK^{(2)}$, has the matrix representation 
\begin{align*}
\cJ_2=\begin{pmatrix}1&0&0\\0&\frac{1}{\sqrt{2}}&0\\0&\frac{1}{\sqrt{2}}&0\\0&0&1\end{pmatrix}
\end{align*}
and $\cJ_2^*$ is given by the transpose of $\cJ_2$. Moreover, $U^{\otimes2}$ can also be written as a matrix. Thus, we get the vector containing the Hermite coefficients of $\cA_{U,2}f_2$ by matrix-vector multiplication, i.e.,
\begin{align*}
\left(\widehat{\cA_{U,2}f_2}(2,0),\widehat{\cA_{U,2}f_2}(1,1),\widehat{\cA_{U,2}f_2}(0,2)\right)^\top\!=\cJ_2^\top U^{\otimes2}\cJ_2\,\left(\widehat{f}(2,0),\widehat{f}(1,1),\widehat{f}(0,2)\right)^\top\!.
\end{align*}
Since it suffices to know the Hermite coefficients of a function, we fully know how the orthogonal transform affects the function $f_2$.
\qed\end{example}

Because of Theorem \ref{th:automorphism} we have that the norm of any function $f\in L^2(\R^d,\varphi_d)$ can not be influenced by applying a mapping $\cA_U$, but we are interested in functions $f$ which belong to a Hermite space $\cH_\r$ with norm $\|\cdot\|_\r$. To preserve the representation \eqref{eq:mapping} of $\cA_U$, we have to equip $\cK$ with a suitable norm. For that, we introduce, for any $m\in\N_0$, on the space $\cK^{(m)}$ the norm
\begin{align}\label{eq:Krnorm}
\|w\|_{\ell^2,\r}=\sum_{\betab\in\{1,\ldots,d\}^m}\r(S(\betab))^{-1}w_\betab^2.
\end{align}
Then it follows that $\cJ$, as defined above, is still a isometry between $\cH_\r$ and $\cK$ as well as $\cJ^*\cJ=\id_{\cH_\r}$. Nevertheless, $\bigoplus_{m=0}^{\infty}U^{\otimes m}$ is not a Hilbert space automorphism on $\cK$ equipped with the norm \eqref{eq:Krnorm}. Consequently, we get that, in general,
\begin{align*}
\|f\|_\r\neq\|\cA_U f\|_\r.
\end{align*}
Thus, we can use orthogonal transforms in the sense that we choose for a given function $f\in\cH_\r$ an orthogonal transform $U$ such that the norm of $\cA_U f$ grows slower in $d$ than $\|f\|_\r$. 

The problem of finding an orthogonal transform that makes the norm 
smaller is not an easy one. However, note that there exists at least one
orthogonal transform such that the transformed problem is not worse than the
original one, namely the identity. 

In Irrgeher and Leobacher \cite{il12, il13} a  ``regression algorithm''
 is introduced
which determines an orthogonal transform tailored to the given problem which is
optimal in some sense. The idea of the algorithm is to approximate the function
$f$ by an linear function using a linear regression approach. For the linear
function it is easy to determine an orthogonal transform so that the linear
function becomes one-dimensional. Then this orthogonal transform is used
together with a QMC rule to solve the original integration
problem. 

It turns out that this algorithm fits well into the setting of Hermite spaces. Let $f$ be in a Hermite space $\cH_r$. Because of Theorem \ref{th:hermitespaceexpansion} we know that $f$ has a pointwise convergent Hermite expansion, i.e.,
\begin{align*}
f(\x)=\sum_{\k\in\N_0^d}\widehat{f}(\k)H_\k(\x)
\end{align*}
for all $\x\in\R^d$. The first step of the regression algorithm is the linear regression approach and it is easy to see that this means that we approximate the function $f$ by using the truncated Hermite expansion of order $1$, i.e.,
\begin{align*}
f(\x)\approx\sum_{\vert\k\vert\leq 1}\widehat{f}(\k)H_\k(\x)\,.
\end{align*}
We denote the multi-index $(0,\ldots,1,\ldots,0)$ with $1$ at the $j$-th entry
by $\e_j$ and we write $\linpart=(\widehat{f}(\e_1),\ldots,\widehat{f}(\e_d))$.
Now we determine an orthogonal transform $U$, namely a Householder reflection,
which maps the vector $(1,0,\ldots,0)\tr$ to
$\linpart/\|\linpart\|_{\ell_2}$, where $\|.\|_{\ell_2}$ denotes the euclidean
norm and where we assume that the linear part $\linpart$ of $f$ does not vanish. So we have that \begin{align*}
f(U\x)=\widehat{f}(\0)+\|\linpart\|_{\ell_2}x_1+\sum_{\vert\k\vert\geq
2}\widehat{f}(k)H_\k(U\x)\,,
\end{align*}
and we see that linear part of the Hermite expansion depends only on $x_1$.
If we now consider the norm of $f\circ U$ we get that
\begin{align*}
\|f\circ U\|^2_r=\widehat{f}(\0)^2+r(\e_1)^{-1}\|\linpart\|_{\ell_2}^2+\sum_{\vert\k\vert\geq 2}r(\k)^{-1}\widehat{f\circ U}(\k)^2\,.
\end{align*}
If $f$ is such that its linear part already covers almost all of the 
$\cH_r$-norm,
it may happen that $U$ reduces significantly the norm, which, in fact, depends
on the function $r$. For example, let us consider $r=p_{\alphab,\gammab}$
with $\gamma_1>\gamma_2\ge \gamma_3\ge \ldots$ .
Then, \begin{align*}
\|f\circ U\|_{p_{\alphab,\gammab}}=\widehat{f}(\0)^2+\gamma_1^{-1}\|\linpart\|_{\ell_2}^2+\sum_{\vert\k\vert\geq 2}p_{\alphab,\gammab}^{-1}(\k)\,\widehat{f\circ U}(\k)^2
\end{align*}
and 
\begin{align*}
\gamma_1^{-1}\|\linpart\|_{\ell_2}^2=\gamma_1^{-1}\sum_{j=1}^{d}\widehat{f}(\e_j)^2< \sum_{j=1}^{d}\gamma_j^{-1}\widehat{f}(\e_j)^2=\sum_{\vert\k\vert=1}p_{\alphab,\gammab}^{-1}(\k)\,|\widehat{f}(\k)|^2\,.
\end{align*}
Thus, we see that the part of the norm  coming from the linear part of $f$
can be reduced by applying the orthogonal transform $U$. That means that, 
under certain conditions on
the influence of the linear fraction of the function as well as on the weighted
Hermite space, the
regression algorithm can strictly reduce the norm of the transformed function .

In Irrgeher and Leobacher \cite{il12} realistic examples coming from finance
are considered. Numerical results show that orthogonal transforms, and in 
particular the ones computed by the 
regression algorithm, also lead to more efficient QMC integration.   

Furthermore, we want to mention that it can happen that a function does not belong to the Hermite space $\cH_\r$, but after applying an orthogonal transform the function is in $\cH_\r$. To illustrate that we consider the following simple example.

\begin{example}
Let $\alpha_1>\alpha_2>1$ and $\gamma_1, \gamma_2>0$. Now we consider function
$f:\R^2\longrightarrow\R$ of the form $f(\x)=f_1(x_1)f_2(x_2)$ with continuous
and Gaussian square-integrable functions $f_1, f_2:\R\longrightarrow\R$.
Furthermore, we suppose that $f_1$ is such that
$\sum_{k=1}^{\infty}k^{\alpha_1}|\hat{f_1}(k)|^2=\infty$ but
$\sum_{k=1}^{\infty}k^{\alpha_2}|\hat{f_1}(k)|^2<\infty$. 
This means that the Hermite
coefficients of $f_1$ decay fast enough such that
$\|f_1\|_{p_{\alpha_2,\gamma_2}}<\infty$, but for the smoothness parameter
$\alpha_1$ the decay is too slow, i.e.\ $\|f_1\|_{p_{\alpha_1,\gamma_1}}=\infty$. For
the function $f_2$ we assume that $0<\sum_{k=1}^{\infty}k^{\alpha_i}|\hat{f_2}(k)|^2<\infty$
holds for $i=1,2$. Then, $f$ does not belong to the Hermite space
$\cH_{\p_{(\alpha_1,\alpha_2),(\gamma_1,\gamma_2)}}$, because 
\begin{align*}
\|f\|_{\p_{(\alpha_1,\alpha_2),(\gamma_1,\gamma_2)}}=\|f_1\|_{p_{\alpha_1,\gamma_1}}\|f_2\|_{p_{\alpha_2,\gamma_2}}=\infty.
\end{align*}
However, let us consider the orthogonal transform $U$, given by 
\begin{align*}
U=\begin{pmatrix}0&1\\1&0\end{pmatrix}.
\end{align*} 
Applying the transform to the function $f$ we have $\cA_Uf(\x)=f(U\x)=f_1(x_2)f_2(x_1)$. Hence,
\begin{align*}
\|\cA_Uf\|_{\p_{(\alpha_1,\alpha_2),(\gamma_1,\gamma_2)}}=\|f_1\|_{p_{\alpha_2,\gamma_2}}\|f_2\|_{p_{\alpha_1,\gamma_1}}<\infty
\end{align*}
and $\cA_Uf\in\cH_{\p_{(\alpha_1,\alpha_2),(\gamma_1,\gamma_2)}}$. Although $f$ itself is not in the Hermite space, we know that the error bound \eqref{eq:error} holds if we apply the orthogonal transform $U$ to the function $f$. 
\end{example}

\section{Conclusion}\label{sec:conclusion}

We have studied the effect of orthogonal transforms on the effectivity
of QMC integration. We have put forward rigorous models in the
framework of weighted norms and reproducing kernel Hilbert spaces. Furthermore,
we have shown how to compute the norm of a function that is concatenated
with an orthogonal transform. This makes it possible to measure
whether the orthogonal transform makes the weighted norm smaller.

Of course, it would be desirable to have an algorithm at hand 
for finding  an optimal (or at least a good)
orthogonal transform for a given function that minimizes the weighted
norm of the transformed function. This has already been tried in
Imai and Tan \cite{imaitan07} as well as in Irrgeher and
Leobacher \cite{il12,il13}, but those approaches rely on linear approximations
of the function and therefore may fail spectacularly for simple but
nonlinear functions.  We have paved the way for more general algorithms
that are similar in spirit but take higher orders of the Hermite 
expansion into account. But a practical algorithm that, for example, proceeds
by approximating the integrand by a quadratic polynomial, has yet to
be developed. 

We have defined and studied multivariate integration in weighted Hermite
spaces. We have succeeded in giving sufficient conditions on the
weight sequence for polynomial tractability and strong polynomial tractability in weighted Hermite
spaces with polynomially decaying coefficients. Necessary conditions have yet
to be found, but we conjecture that they might coincide with the sufficient
ones. Moreover, we have given necessary and sufficient conditions
on polynomial tractability and strong polynomial tractability for a class of weighted Hermite spaces
with exponentially decaying coefficients.

The question of finding concrete point sets for QMC integration in the 
spaces considered is completely open. Likely candidates are maps of classical
low-discrepancy sequences in the unit hypercube to the $\R^d$ via the inverse
cumulative distribution function of the standard normal distribution. 

In any case we have contributed to the problem of explaining
the efficiency of QMC integration for very high-dimensional problems
by providing a general and flexible framework, namely that of Hermite spaces
and orthogonal transforms,
and by giving important first results on the subject in this framework.

\section*{Acknowledgments}
The authors are supported by the Austrian Science Fund (FWF): Projects F5508-N26 (Leobacher) and F5509-N26 (Irrgeher), respectively. The projects F5508-N26 and F5509-N26 are part of the Special Research Program ``Quasi-Monte Carlo Methods: Theory and Applications''. The first author is also partially supported by the Austrian Science Foundation (FWF), Project P21943.

\end{document}